\documentclass[12pt]{amsart}
\usepackage[english]{babel}
\usepackage{graphicx}
\usepackage{amsmath,amssymb,hyperref,srcltx}
\newcommand{\diff}[2]{\mbox{{\rm Diff}{${\,}_{#1}({\mathbb C}^{#2},0)$}}}
\newcommand{\diffh}[2]{\mbox{$\widehat{\rm Diff}{{\,}_{#1}({\mathbb C}^{#2},0)}$}}

\newcommand{\cn}[1]{\mbox{(${\mathbb C}^{#1},0$)}}

\newtheorem{pro}{Proposition}[section]
\newtheorem{teo}{Theorem}[section]

\newtheorem{cor}{Corollary}[section]

\newtheorem{lem}{Lemma}[section]

\theoremstyle{remark}
\newtheorem{rem}{Remark}[section]

\theoremstyle{definition}
\newtheorem{defi}{Definition}[section]

\begin{document}

\title[Dimension of groups of diffeomorphisms]{Finite dimensional groups of local diffeomorphisms}

\author{Javier Rib\'{o}n}
\address{Instituto de Matem\'{a}tica, UFF, Rua M\'{a}rio Santos Braga S/N
Valonguinho, Niter\'{o}i, Rio de Janeiro, Brasil 24020-140}
\thanks{e-mail address: javier@mat.uff.br}
\thanks{MSC class. Primary: 32H50, 37C85; Secondary: 37F75,  20G20, 20F16}
\thanks{Keywords: local diffeomorphism, solvable group, nilpotent group}
%\date{\today}
\maketitle

\bibliographystyle{plain}
\section*{Abstract}
We are interested in classifying groups of local biholomorphisms (or even formal
diffeomorphisms) that can be endowed with a canonical structure of algebraic group up
to add extra formal diffeomorphisms.
We show that this is the case for  virtually polycyclic subgroups and in particular
finitely generated virtually nilpotent groups of local biholomorphisms.
We provide several methods to identify this property and build examples.

As a consequence we generalize results of Arnold, Seigal-Yakovenko and Binyamini
on uniform estimates of local intersection multiplicities to bigger classes of groups, including for example
virtually polycyclic groups.
\section{Introduction}
We study the action of groups of self-maps on intersection multiplicities.
More precisely, given varieties $V$ and $W$ of complementary dimension of
an ambient space $M$ and a subgroup $G$ of self-maps of $M$,
we want to identify conditions guaranteeing that $F \mapsto (F(V), W)$
is bounded as a function of $G$.
%
%We are interested in studying intersection multiplicities $(F(V), W)$ where
%$V,W$ are varieties of complementary dimension of an ambient space $M$
%and $F$ varies in a subgroup $G$ of continuous self-maps of $M$.
%More precisely, we want to identify conditions guaranteeing that $(F(V), W)$
%is bounded as a function of $G$.
%
Let us introduce a classical example of an application of such a property.
%explain why such property is interesting.
Consider a continuous map $F:M \to M$ and an isolated fixed point $P$ of $F$.
The fixed point index of $F$ at $P$ is equal to the topological intersection index of
$\Delta$ and $(F \times Id)(\Delta)$ at $(P,P)$ where $\Delta$ is the diagonal of
$M \times M$.
By considering the iterates $(F \times Id)^{n}$ with $n \in {\mathbb Z}$ we obtain
fixed point indexes for the fixed points of the iterates of $F^{n}$, i.e. for periodic points.
In the context of $C^{1}$ maps Shub and Sullivan proved that the intersection
index of $\Delta$ and $(F \times Id)^{n}(\Delta)$ at isolated fixed points is uniformly bounded.
\begin{teo}[\cite{shub-sullivan}]
Let  $U$ be an open subset of ${\mathbb R}^{m}$. Let $F: U \to {\mathbb R}^{m}$ be a
$C^{1}$ map such that $0$ is an isolated fixed point of $F^{n}$ for any $n \geq 1$.
Then the fixed point index of $F^{n}$ at $0$  is bounded as a function of $n$.
\end{teo}
As an immediate corollary they show that a $C^1$ map $F:M \to M$
defined in a compact differentiable manifold $M$ has
infinitely many periodic points if
the sequence of Lefschetz numbers $(L(F^{n}))_{n \geq 1}$ is unbounded.

 We denote by $\diff{}{n}$ the group of germs of biholomorphism defined in a neighborhood of
the origin in ${\mathbb C}^{n}$. 
We are interested in 
%these 
uniform intersection results in the local holomorphic setting.
%In this paper we study intersection properties for subgroups of $\diff{}{n}$. 
More precisely,
we want to identify subgroups $G$ of $\diff{}{n}$ satisfying that the set
\[ \{ (\phi(V), W) : \phi \in G \ \mathrm{and} \  (\phi(V), W) < \infty \} \] 
of intersection multiplicities (cf. Definition \ref{def:intmul})
is bounded for any pair of germs of holomorphic varieties $V, W$ defined
in a neighborhood of $0$ in ${\mathbb C}^{n}$.
%of complementary dimension.

The first result in this direction is due to Arnold.
\begin{teo}[{\cite[Theorem 1]{Arnold:intersection}}]
\label{teo:a}
Let $\phi \in \diff{}{n}$. Consider germs  of submanifolds $V,W$ of $({\mathbb C}^{n},0)$
of complementary dimension. Suppose that the intersection multiplicity $\mu_{n}:= (\phi^{n}(V),W)$ is finite for
any $n \in {\mathbb Z}$. Then the sequence $(\mu_{n})_{n \in {\mathbb Z}}$ is bounded.
\end{teo}
The proof is a consequence of the Skolem-Mahler-Lech theorem on roots of quasipolynomials
\cite{skolem:seq}.

The previous result was generalized to the finitely generated abelian case by Seigal and Yakovenko.
We denote by $\diffh{}{n}$ the group of formal diffeomorphisms (cf. Definition \ref{def:fordif}).
\begin{teo}[{\cite[Theorem 1]{Seigal-Yakovenko:ldi}}]
\label{teo:sy}
Let $G$ be an abelian subgroup of $\diffh{}{n}$ generated by finitely many cyclic and one parameter groups.
Consider formal subvarieties $V$, $W$. 
%of complementary dimensions. 
Then the set 
\[ \{ (\phi(V), W) : \phi \in G \ \mathrm{and} \  (\phi(V), W) < \infty \} \] 
is bounded.
%$C>0$ such that
%\[ (\phi(V), W) < \infty \implies  (\phi(V), W) < C \]
%for any $\phi \in G$.
\end{teo}
The group $\diff{}{n}$ is a subgroup of $\diffh{}{n}$ and hence Theorem \ref{teo:sy}
holds for subgroups of $\diff{}{n}$ and germs of subvarieties $V$ and $W$.
In contrast with Theorem \ref{teo:a}
notice that it is not necessary to require that all intersection multiplicities are finite.
The proof relies on a noetherianity argument (cf. section \ref{sec:locint}).

An analogous result was proved by Binyamini for the case in which the subgroup $G$ of
$\diffh{}{n}$ is embedded in a group of formal diffeomorphisms that has a natural
Lie group structure.
\begin{teo}[{\cite[Theorem 5]{Binyamini:finite}}]
\label{teo:b}
Let $G$ be a Lie subgroup (cf. Definition \ref{def:lie})
of $\diffh{}{n}$ with finitely many connected components.
Consider formal subvarieties $V$, $W$. 
%of complementary dimensions. 
Then the set 
\[ \{ (\phi(V), W) : \phi \in G \ \mathrm{and} \  (\phi(V), W) < \infty \} \] 
is bounded.
%Then there exists
%$C>0$ such that
%\[ (\phi(V), W) < \infty \implies  (\phi(V), W) < C \]
%for any $\phi \in G$.
\end{teo}
Theorem \ref{teo:sy} has more natural hypotheses (commutativity and finite generation)
but Theorem \ref{teo:b} is somehow more general. Indeed Binyamini shows that any finitely generated
abelian subgroup of $\diffh{}{n}$ is a subgroup of a Lie group with finitely many connected
components \cite{Binyamini:finite}. Thus it is interesting to study how to find an extension of a
subgroup of $\diffh{}{n}$ that is also a Lie group.
In this paper we characterize the subgroups
of $\diffh{}{n}$ that can be embedded  in a Lie group
(with finitely many connected components) in a natural way.
Moreover we show that every such a group can be canonically
embedded in an algebraic matrix group.
We call these groups {\it finite dimensional}.

Let us be more precise.
%Let us explain why finite dimensionality is an interesting concept for subgroups of $\diffh{}{n}$.
%First, it is canonical;
We define a Zariski-closure $\overline{G}$ 
of a subgroup $G$ of $\diffh{}{n}$ (cf. Definition \ref{def:zarclos}); it
%Such group $\overline{G}$
is a projective limit of algebraic matrix groups and hence it has a
natural definition of dimension. We will say that $G$ is finite dimensional if $\overline{G}$ is finite
dimensional (cf. Definition \ref{def:dim}).
If $G$ is finite dimensional then $\overline{G}$ is isomorphic to one of its subgroups of 
$k$-jets and
hence $\overline{G}$ can be interpreted as an algebraic group (Proposition \ref{pro:elem}).
Equivalently the group $G$ is finite dimensional if and only if there exists $k_{0} \in {\mathbb N}$ such that 
the coefficients of degree greater than $k_0$ in the Taylor expansion at the origin of the elements 
of $G$ are polynomial functions on the coefficients of degree less or equal than $k_0$ 
(Remark \ref{rem:hotlot}).

Finite dimensional subgroups satisfy uniform local intersection properties.
\begin{teo}
\label{teo:main}
Let $G$ be a  finite dimensional subgroup of $\diffh{}{n}$.
Consider ideals $I$ and $J$ of $ \hat{\mathcal O}_{n}$.
%of complementary pure dimensions.
Then the set 
\[ \{ (\phi^{*}(I), J) : \phi \in G \ \mathrm{and} \ (\phi^{*}(I), J) < \infty \} \]
is bounded.
%Then there exists  $C>0$ such that
%\[ (\phi^{*}(I), J) < \infty \implies  (\phi^{*}(I), J) < C \]
%for any $\phi \in G$.
\end{teo}
We define $\hat{\mathcal O}_{n}$ as the ring of formal power series with complex
coefficients in $n$ variables.

Since an algebraic group is a complex Lie group with finitely many connected components,
Theorem \ref{teo:main} can be seen as a consequence of Binyamini's Theorem \ref{teo:b}.
Anyway, the finite dimensional hypothesis provides a simplification of the proof.

We will exhibit different methods to find finite dimensional subgroups of $\diffh{}{n}$
(cf.  Theorems \ref{teo:fw}, \ref{teo:ext}...).
In particular we
identify several algebraic group properties implying that a subgroup of $\diffh{}{n}$ is finite
dimensional, including some notable ones.  Our main result is the following theorem:
\begin{teo}
\label{teo:mainc}
Let $G$ be a subgroup of $\diffh{}{n}$ such that either is
\begin{itemize}
\item virtually polycyclic or
\item virtually nilpotent and generated by finitely many cyclic and one parameter subgroups of $G$.
\end{itemize}
Then $G$ is finite dimensional. In particular 
%of complementary pure dimensions.
the set 
\[ \{ (\phi^{*}(I), J) : \phi \in G \ \mathrm{and} \ (\phi^{*}(I), J) < \infty \} \]
is bounded for any pair of ideals $I$ and $J$ of $ \hat{\mathcal O}_{n}$.
%Then there exists  $C>0$ such that
%\[ (\phi^{*}(I), J) < \infty \implies  (\phi^{*}(I), J) < C \]
%for any $\phi \in G$.
\end{teo}
Check Definitions \ref{def:vir}, \ref{def:norp} and \ref{def:nil} out
for the definitions of virtual property, polycyclic and nilpotent groups respectively.
%
%The subject of this paper is characterizing when a group of local biholomorphisms can
%be embedded in an algebraic group and its consequences in local intersection theory.
%
%We will show a more general result than Theorem \ref{teo:mainc}, where the group $\diff{}{n}$ is replaced by
%the group $\diffh{}{n}$ of formal diffeomorphisms (cf. Definition \ref{def:fordif})
%and $V, W$ are ideals of the ring $\hat{\mathcal O}_{n}:={\mathbb C}[[z_1,\ldots,z_n]]$,
%of complex formal power series in $n$ variables, of complementary pure dimension.
%

Notice that in particular Theorem \ref{teo:mainc} applies to finitely generated virtually nilpotent
subgroups of $\diffh{}{n}$, i.e. to subgroups of polynomial growth of formal diffeomorphisms.

We introduce several techniques that allow to build finite dimensional subgroups of $\diffh{}{n}$.
Every time that we identify such a group we obtain an analogue of Theorem \ref{teo:mainc}.
Instead of writing down the most general possible result,
we prefer to highlight  some remarkable properties that imply finite dimensionality.

Let us compare the hypotheses of Theorems \ref{teo:b} and \ref{teo:main}.
A priori it could be possible for a subgroup $G$ of $\diffh{}{n}$ to satisfy
$\dim G = \infty$ and being the image by a morphism of a real Lie group with finitely many
connected components (cf. Definition \ref{def:lie}). We will see that it never happens (Theorem \ref{teo:ext})
and thus our canonical
approach encloses the results by
Seigal-Yakovenko  and Binyamini \cite{Seigal-Yakovenko:ldi, Binyamini:finite}.
%A priori it is more restrictive for a subgroup of $\diffh{}{n}$ to be finite dimensional than being embedded in
%a Lie  group with finitely many connected components (cf. Definition \ref{def:lie}) but we will see that the
%latter groups are always finite dimensional (Theorem \ref{teo:ext}).
The definition of finite dimensional
subgroup of $\diffh{}{n}$ allows us to apply the techniques of the algebraic group theory
%to groups of local diffeomorphisms
in the study of local intersection problems.

Our canonical approach makes simpler to analyze whether or not subgroups of $\diffh{}{n}$
are embeddable in algebraic groups. We will relate finite dimensionality with other group properties, namely
\begin{itemize}
\item Finite determination (cf. Definition \ref{def:findet}) properties.
%Given a subgroup $G$ of $\diffh{}{n}$,
We will show
that finite determination implies finite dimension under certain closedness properties (Corollary \ref{cor:charfdifd}). \\
\item Finite decomposition properties. A subgroup $G$ of $\diffh{}{n}$ is finite dimensional if and only if
every element can be written as a word of uniformly bounded length in an alphabet whose letters belong 
to the union of finitely many cyclic and one parameter subgroups of $\diffh{}{n}$ (Theorem \ref{teo:fw} and
Remark \ref{rem:fw}). \\
\item Virtually solvable subgroups of $\diffh{}{n}$ with suitable finite generation hypotheses are finite dimensional
(Proposition \ref{pro:fd1fd}, Theorems \ref{teo:vp}, \ref{teo:vnfgifd}, Corollaries \ref{cor:vsfg}, \ref{cor:vssp}...) \\
\item Decomposition of the group in a tower of extensions of the trivial group. \\
\end{itemize}
%As an example
%the classes of groups in the hypothesis of Theorem \ref{teo:mainc} are finite dimensional by
%Theorems \ref{teo:vp} and \ref{teo:vnfgifd} respectively.
%
Let us expand on the final item of the previous list.
Consider a normal subgroup $H$ of a subgroup $G$ of $\diffh{}{n}$.
We can define the codimension of $H$ or equivalently the dimension of the extension
$G/H$ as the codimension of $\overline{H}$ in $\overline{G}$.
Such definition is interesting because a tower of finite dimensional extensions is finite dimensional
(Proposition \ref{pro:elems}). In particular it is possible to decide whether or not a subgroup of $\diffh{}{n}$
is finite dimensional
by considering it as a tower of (easier to handle) extensions of the trivial group.  We will exhibit
some classes of extensions that are finite dimensional, namely
\begin{itemize}
\item Finite extensions. \\
\item Finitely generated abelian extensions. \\
\item $G/H$ is a  connected Lie group (cf. Definition \ref{def:lie}). \\
\item Certain subextensions of virtually solvable extensions (Theorems \ref{teo:subus} and \ref{teo:rgwgu}).
\end{itemize}
For instance a virtually polycyclic group is a tower of cyclic and finite extensions of 
the trivial group and hence it is always finite dimensional.  
Hence 
%in such a case 
we can apply Theorem \ref{teo:main} to show Theorem \ref{teo:mainc}.
%In this way we can obtain Theorem \ref{teo:mainc}
%for finitely generated virtually nilpotent groups since finitely generated nilpotent groups are
%polycyclic. 
The extension approach provides a method to build examples of finite dimensional subgroups
of $\diffh{}{n}$.
\section{Pro-algebraic groups}
Let us explain some of the basic properties of the pro-algebraic subgroups of
$\diffh{}{n}$. Pro-algebraic groups of formal diffeomorphisms have been used
in the study of differential Galois theory by
Morales-Ramis-Sim\'{o} \cite{Mor-Ram-Simo}.
Most of the results in this section can be found in \cite{JR:arxivdl} and
\cite{JR:solvable25}. We explain them here for the sake of clarity and completeness.
\subsection{Formal vector fields and diffeomorphisms}
\label{sec:fvfd}
Let us introduce some notations.
\begin{defi}
We denote by ${\mathcal O}_{n}$ the ring ${\mathbb C}\{z_{1},\ldots,z_{n}\}$
of germs of holomorphic functions defined in the neighborhood of $0$ in ${\mathbb C}^{n}$.
We denote by ${\mathfrak m}$ the maximal ideal of ${\mathcal O}_{n}$.

Analogously we define $\hat{\mathcal O}_{n}$ as the ring of formal power series with complex
coefficients in $n$ variables whose maximal ideal will be denoted by $\hat{\mathfrak m}$.
\end{defi}
Next,  we define formal vector fields as a generalization of local vector fields.
\begin{defi}
We denote by $\mathfrak{X} ({\mathbb C}^{n},0)$ the Lie algebra of germs of holomorphic
vector fields defined in the neighborhood of $0$ in ${\mathbb C}^{n}$ that are singular at
$0$.
\end{defi}
\begin{rem}
An element $X$ of $\mathfrak{X} ({\mathbb C}^{n},0)$ is of the form
\[ X= f_{1}(z_1,\ldots,z_n) \frac{\partial}{\partial z_1} + \ldots +
 f_{n}(z_1,\ldots,z_n) \frac{\partial}{\partial z_n}  \]
where $f_{1}, \ldots, f_{n}$ belong to the maximal ideal ${\mathfrak m}$ of ${\mathcal O}_{n}$.
Analogously $X$ can be interpreted as a derivation of the ${\mathbb C}$-algebra ${\mathfrak m}$.
\end{rem}
\begin{defi}
We define the Lie algebra $\hat{\mathfrak{X}} ({\mathbb C}^{n},0)$ as the set of
derivations of the ${\mathbb C}$-algebra $\hat{\mathfrak m}$. Analogously we can
identify an element $X$ of $\hat{\mathfrak{X}} ({\mathbb C}^{n},0)$ with the
expression
\[ X = X(z_1)  \frac{\partial}{\partial z_1} + \ldots +  X(z_n)  \frac{\partial}{\partial z_n} \]
where the coefficients of the vector field belong to $\hat{\mathfrak m}$.
\end{defi}
Let us apply the same program to diffeomorphisms.
%We can define also formal diffeomorphisms.
\begin{defi}
We denote by $\diff{}{n}$ the group of germs of biholomorphism defined in the neighborhood of
$0$ in ${\mathbb C}^{n}$.
\end{defi}
\begin{rem}
An element $\phi$ of $\diff{}{n}$ is of the form
\[ \phi(z_{1},\ldots,z_{n}) = (f_{1}(z_1,\ldots,z_n), \ldots, f_{n}(z_1,\ldots,z_n)) \]
where $f_1, \ldots, f_n \in {\mathfrak m}$ and its linear part $D_{0} \phi$ at the origin 
is an invertible linear map.
\end{rem}
\begin{defi}
\label{def:fordif}
We say that $\phi$ belongs to the group $\diffh{}{n}$ of formal diffeomorphisms if it is of the form
\[ \phi(z_{1},\ldots,z_{n}) = (f_{1}(z_1,\ldots,z_n), \ldots, f_{n}(z_1,\ldots,z_n)) \]
where $f_1, \ldots, f_n \in \hat{\mathfrak m}$ and $D_{0} \phi$ is an invertible linear map.
\end{defi}
We will use the Krull topology (the $\hat{\mathfrak m}$-adic topology) in our spaces of formal
objects.
\begin{defi}
The sets of the form $f + \hat{\mathfrak m}^{j}$ for any choice of  $f \in \hat{\mathcal O}_{n}$
and $j \geq 0$ are a base of open sets of a topology in $\hat{\mathcal O}_{n}$,
the so called $\hat{\mathfrak m}$-{\it adic} (or {\it Krull}) {\it topology}.
Since we can interpret formal vector fields and diffeomorphisms as $n$-uples of elements
in $\hat{\mathfrak m}$ we can define the Krull topology in
$\hat{\mathfrak X} ({\mathbb C}^{n},0)$ and $\diffh{}{n}$.
\end{defi}
\begin{rem}
A sequence $(f_{k})_{k \geq 1}$ of elements of $\hat{\mathcal O}_{n}$
converges to $f \in \hat{\mathcal O}_{n}$ in the Krull topology if
for any $j \in {\mathbb N}$ there exists $k_{0} \in {\mathbb N}$ such that
$f - f_k \in \hat{\mathfrak m}^{j}$ for any $k \geq k_{0}$.
Convergence of sequences in $\hat{\mathfrak X} ({\mathbb C}^{n},0)$ and $\diffh{}{n}$
is analogous.
\end{rem}
\begin{rem}
It is clear that $\hat{\mathcal O}_{n}$, $\hat{\mathfrak X} ({\mathbb C}^{n},0)$ and $\diffh{}{n}$
are the closures in the Krull topology of
${\mathcal O}_{n}$, ${\mathfrak X} ({\mathbb C}^{n},0)$ and $\diff{}{n}$
respectively.
\end{rem}

It is difficult to work with $\diffh{}{n}$ since it is an infinite dimensional space.
Anyway we can understand $\diffh{}{n}$ as a projective limit
$\varprojlim_{k \in {\mathbb N}} D_{k}$ where every $D_{k}$ is a finite dimensional matrix
group for $k \in {\mathbb N}$  \cite[Lemma 2.2]{JR:solvable25}.
We should interpret $D_{k}$ as the group of $k$-jets of elements of $\diffh{}{n}$.
Next, let us explain how to define rigorously the groups $D_k$ for $k \in {\mathbb N}$
and how this allows to
apply the theory of linear algebraic groups to the groups of formal diffeomorphisms.

Given $X \in \hat{\mathfrak X}({\mathbb C}^{n},0)$ and $\phi \in \diffh{}{n}$ we can associate
$X_{k}, \phi_{k} \in \mathrm{GL}({\mathfrak m}/ {\mathfrak m}^{k+1})$
respectively for any $k \in {\mathbb N}$. They are  given by
\[
\begin{array}{ccc}
\hat{\mathfrak m}/ \hat{\mathfrak m}^{k+1} & \stackrel{X_{k}}{\to}
& \hat{\mathfrak m}/ \hat{\mathfrak m}^{k+1} \\
f + \hat{\mathfrak m}^{k+1} & \mapsto & X(f)+ \hat{\mathfrak m}^{k+1},
\end{array} \ \ \
\begin{array}{ccc}
\hat{\mathfrak m}/ \hat{\mathfrak m}^{k+1} & \stackrel{\phi_{k}}{\to}
& \hat{\mathfrak m}/ \hat{\mathfrak m}^{k+1} \\
f + \hat{\mathfrak m}^{k+1} & \mapsto & f  \circ \phi+ \hat{\mathfrak m}^{k+1}.
\end{array}
\]
The linear map $X_{k}$ (resp. $\phi_k$)
determines and is determined by the $k$-jet of $X$ (resp. $\phi$).
Moreover $L_{k}:= \{ X_{k} : X \in \hat{\mathfrak X}({\mathbb C}^{n},0) \}$ is the Lie algebra
of the group $D_k := \{ \phi_k : \phi \in \diffh{}{n} \}$.
It is an algebraic subgroup of
$\mathrm{GL}(\hat{\mathfrak m}/ \hat{\mathfrak m}^{k+1})$ since it satisfies
\[ D_k = \{ \alpha \in \mathrm{GL}(\hat{\mathfrak m}/ \hat{\mathfrak m}^{k+1}):
\alpha (f g) = \alpha (f) \alpha (g) \ \forall f,g \in \hat{\mathfrak m}/ \hat{\mathfrak m}^{k+1}), \]
cf. \cite[section 3]{JR:arxivdl} \cite[Lemma 2.1]{JR:solvable25}. 
\begin{defi}
\label{def:trunc}
Given $k \geq l \geq 1$
we define the maps $\pi_{k}: \diffh{}{n} \to D_{k}$ and  $\pi_{k,l}:D_{k} \to D_{l}$ given by
$\pi_k (\phi)= \phi_k$ and $\pi_{k,l}(\phi_{k})= \phi_{l}$.
\end{defi}
Since $D_k$ is the group of truncations of elements of
$\diffh{}{n}$ up to level $k$, we can interpret $\diffh{}{n}$ as
the projective limit of the projective system
$(\varprojlim_{k \in {\mathbb N}} D_{k}, (\pi_{k,l})_{k \geq l \geq 1})$
of algebraic groups and morphisms of algebraic groups \cite[Lemma 2.2]{JR:solvable25}.
Analogously $\hat{\mathfrak X}({\mathbb C}^{n},0)$ is the projective limit
$\varprojlim L_{k}$.
\subsection{Exponential map}
Given $X \in \hat{\mathfrak X} ({\mathbb C}^{n},0)$ we can define its
exponential $\mathrm{exp}(X)$. Indeed given 
$(X_{k})_{k \geq 1} \in \varprojlim L_{k} = \hat{\mathfrak X}({\mathbb C}^{n},0)$
the family
$(\mathrm{exp}(X_{k}))_{k \geq 1}$ defines an element
$\mathrm{exp}(X)$ of $\diffh{}{n} = \varprojlim D_{k}$.
Equivalently we consider a sequence
$(X_{j})_{j \in {\mathbb N}}$ of convergent vector fields that converges to $X$ in the
Krull topology and then we define $\mathrm{exp}(X)$ as the limit in the Krull topology
of $(\mathrm{exp}(X_{j}))_{j \in {\mathbb N}}$
where $\mathrm{exp}(X_j)$ is the time $1$ flow of $X_j$ for $j \in {\mathbb N}$.
%The definition does not depend on the choice of $(X_{j})_{j \in {\mathbb N}}$.
\begin{defi}
We say that a formal vector field $X \in \hat{\mathfrak X} ({\mathbb C}^{n},0)$ is {\it nilpotent}
if its linear part $D_{0} X$ is nilpotent. We denote by $\hat{\mathfrak X}_{N} ({\mathbb C}^{n},0)$
the subset of $\hat{\mathfrak X} ({\mathbb C}^{n},0)$ of formal nilpotent vector fields.
\end{defi}
\begin{defi}
We say that a formal diffeomorphism $\phi \in \diffh{}{n}$ is {\it unipotent}
if its linear part $D_{0} \phi$ is unipotent. We denote by $\diffh{u}{n}$
the subset of $\diffh{}{n}$ of formal unipotent diffeomorphisms.
We say that a subgroup $G$ of $\diffh{}{n}$ is unipotent if $G \subset \diffh{u}{n}$.
\end{defi}
\begin{pro}[cf. {\cite{Ecalle, MaRa:aen}} {\cite[Th. 3.17]{Ilya-Yako}}]
The map 
\[ \mathrm{exp}: \hat{\mathfrak X}_{N} ({\mathbb C}^{n},0) \to \diffh{u}{n} \]
is a bijection.
\end{pro}
\begin{defi}
Given $\phi \in \diffh{u}{n}$ we define its {\it infinitesimal generator} $\log \phi$ as the unique
formal nilpotent vector field such that $\phi = \mathrm{exp}(\log \phi)$.
We denote $\phi^{t} =  \mathrm{exp}(t \log \phi)$ for $t \in {\mathbb C}$.
\end{defi}

\subsection{Zariski-closure of a group of formal diffeomorphisms}
\begin{defi}
Let $G$ be a subgroup of $\diffh{}{n}$. Given $k \in {\mathbb N}$ we define
$G_{k}^{*} = \{ \phi_k : \phi \in G \}$ and $G_{k}$ as the Zariski-closure of
$G_{k}^{*}$  in
$\mathrm{GL}(\hat{\mathfrak m}/ \hat{\mathfrak m}^{k+1})$. 
\end{defi}
Let us remark that
since $D_k$ is algebraic, $G_k$ is a subgroup of $D_k$ for any $k \in {\mathbb N}$.
\begin{rem}
\label{rem:pik}
Given $k \geq l \geq 1$ the image  of the algebraic closure
of $G_{k}^{*}$ by $\pi_{k,l}$ is the algebraic closure of the image
$G_{l}^{*}$ (cf. \cite[2.1 (f), p. 57]{Borel}).
Hence we obtain $\pi_{k,l}(G_{k})= G_{l}$ for any $k \geq l \geq 1$.
In particular $(\varprojlim_{k \in {\mathbb N}} G_{k}, (\pi_{k,l})_{k \geq l \geq 1})$ is
a projective system.
\end{rem}
\begin{defi}
\label{def:zarclos}
Let $G$ be a subgroup of $\diffh{}{n}$. We define the Zariski-closure $\overline{G}$
of $G$ as $\varprojlim_{k \in {\mathbb N}} G_{k}$ or in other words
\[ \overline{G} = \{ \phi \in \diffh{}{n} : \phi_{k} \in G_{k} \ \forall k \in {\mathbb N} \} . \]
\end{defi}
\subsection{Definition of pro-algebraic group}
\begin{defi}
Let $G$ be a subgroup of $\diffh{}{n}$. We say that $G$ is {\it pro-algebraic}
if $G= \overline{G}$.
\end{defi}
\begin{rem}
\label{rem:pik2}
Since $\pi_{k,l}(G_{k})= G_{l}$ for any $k \geq l \geq 1$,
the natural projection
$(\pi_k)_{|\overline{G}}: \overline{G} \to G_{k}$ is surjective for any $k \in {\mathbb N}$
(cf. \cite[Lemma 2.5]{JR:solvable25}, \cite[Corollary 3.25]{rib:cimpa}).
Thus  the Zariski-closure of
$\overline{G}$ coincides with $\overline{G}$ and
$\overline{G}$ is pro-algebraic. It is the minimal pro-algebraic group containing $G$.
\end{rem}
We can characterize pro-algebraic subgroups of $\diffh{}{n}$.
\begin{pro}[{\cite[Proposition 2.2]{JR:solvable25}}, cf. {\cite[Proposition 3.26]{rib:cimpa}}]
\label{pro:char}
Let $G$ be a subgroup of $\diffh{}{n}$.
Then $G$ is pro-algebraic if and only if $G_{k}^{*}$
is an algebraic subgroup of
$\mathrm{GL}(\hat{\mathfrak m}/ \hat{\mathfrak m}^{k+1})$ for any $k \in {\mathbb N}$
and $G$ is closed in the Krull topology.
\end{pro}
\subsection{Lie algebra of a pro-algebraic group}
Pro-algebraic groups have a connected component of $Id$ whose properties are analogous
to the connected component of $Id$ of an algebraic matrix group. This is a particular
instance of a more general situation:
many analogues of concepts involving algebraic groups can be transferred to the
pro-algebraic setting.
\begin{defi}
Let $G$ be a subgroup of $\diffh{}{n}$.
We define $G_{k,0}$ as the connected component of $Id$ of $G_{k}$ for $k \in {\mathbb N}$.
We define $\overline{G}_{0}  = \varprojlim_{k \in {\mathbb N}} G_{k,0}$ or equivalently
\[ \overline{G}_{0}  =
\{ \phi \in \overline{G} : \phi_{k} \in G_{k,0} \ \forall k \in {\mathbb N} \}. \]
We say that $\overline{G}_{0}$ is the
{\it connected component of} $Id$ of $\overline{G}$.
If $G$ is pro-algebraic then we denote $G_{0} = \overline{G}_{0}$.
\end{defi}
\begin{rem}
\label{rem:ccia}
The group $\overline{G}_{0}$ is a finite index normal
pro-algebraic subgroup of $\overline{G}$
\cite[Proposition 2.3 and Remark 2.9]{JR:solvable25}, cf.
\cite[Proposition 3.35 and Remark 3.37]{rib:cimpa}.
\end{rem}
\begin{pro}[{\cite[Proposition 2]{JR:arxivdl}}]
\label{pro:liecorr}
Let $G$ be a subgroup of $\diffh{}{n}$.
We consider
\[ {\mathfrak g} = \{ X \in \hat{\mathfrak X}  ({\mathbb C}^{n},0) :
\mathrm{exp}(t X) \in \overline{G} \ \forall t \in {\mathbb C} \}. \]
Then ${\mathfrak g}$ is a Lie algebra and $\overline{G}_{0}$ is generated by
the set $\mathrm{exp} ({\mathfrak g})$.
\end{pro}
\begin{defi}
We say that ${\mathfrak g}$ is the Lie algebra of $\overline{G}$.
\end{defi}
It is natural to consider $\overline{G}_{0}$ as the connected component of $Id$ of
$\overline{G}$ since it is a finite index normal subgroup of $\overline{G}$
that is generated by the exponential of the Lie algebra of $\overline{G}$.

The Zariski-closure of a cyclic subgroup of $\diffh{u}{n}$ is connected and one dimensional.
\begin{rem}[{\cite[Remark 2.11]{JR:solvable25}}, cf. {\cite[Remark 3.30]{rib:cimpa}}]
\label{rem:closu}
Let $\phi$ be a unipotent element of $\diffh{}{n}$.
Then $\overline{ \langle \phi \rangle}$ is equal to
$\{ \phi^{t} : t \in {\mathbb C} \}$.
In particular the Lie algebra of $\overline{ \langle \phi \rangle}$ is the one dimensional 
complex vector space generated by $\log \phi$.
%$\log \phi$ belongs to the Lie algebra of $\overline{G}$.
\end{rem}

The next property is well-known in the finite dimensional setting.
\begin{lem}
\label{lem:fiscg}
Let $H$ be a finite index subgroup of a pro-algebraic subgroup $G$ of $\diffh{}{n}$.
Then $H$ contains $G_{0}$.
\end{lem}
\begin{proof}
Given an element $X$ in the Lie algebra ${\mathfrak g}$ of $G$, its one-parameter
group $\{ \mathrm{exp}(t X): t \in {\mathbb C} \}$ is contained in $G$.
Since $H$ is a finite index subgroup of $G$,
$\{ \mathrm{exp}(t X): t \in {\mathbb C} \}$ is contained in $H$.
Any element of $G_{0}$ is of the form
$\mathrm{exp}(X_{1}) \circ \ldots \circ \mathrm{exp}(X_{m})$ for some
$X_{1}, \ldots, X_{m} \in {\mathfrak g}$ by Proposition \ref{pro:liecorr}.
Hence $G_{0}$ is contained in $H$.
\end{proof}
The next result will be used later on to identify pro-algebraic groups.
\begin{teo}[Chevalley, cf. {\cite[section I.2.2, p. 57]{Borel}}]
\label{teo:chevc}
The group generated by a family of connected algebraic matrix groups is algebraic.
\end{teo}
\subsection{Normal subgroups of pro-algebraic groups}
Next results relate the properties of normal subgroups with those of their algebraic
closures.
\begin{lem}
\label{lem:clnin}
Let $H$  be a normal subgroup of a subgroup $G$ of $\diffh{}{n}$.
Then $\overline{H}$ is a normal subgroup of $\overline{G}$.
Moreover $H_{k}$ is a normal subgroup of $G_{k}$ for any
$k \in {\mathbb N}$.
\end{lem}
\begin{proof}
Fix $k \in {\mathbb N}$.
We have $A  H_{k}^{*} A^{-1} = H_{k}^{*}$ for any
$A \in G_{k}^{*}$. We deduce
$A H_{k} A^{-1} = H_{k}$ for any $A \in G_{k}^{*}$.
The normalizer of the algebraic subgroup $H_{k}$ in the algebraic group
$G_{k}$ is algebraic and contains $G_{k}^{*}$. Hence
it is equal to $G_{k}$ and then $H_{k}$ is a normal subgroup of $G_{k}$ for
any $k \in {\mathbb N}$. As a consequence
$\overline{H}$ is a normal subgroup of $\overline{G}$.
\end{proof}
\begin{lem}
\label{lem:profis}
Let $H$ be a finite index subgroup of a subgroup $G$ of $\diffh{}{n}$.
Then $H$ is pro-algebraic if and only if $G$ is pro-algebraic.
\end{lem}
\begin{proof}
Since $G_{k}^{*}$ and $H_{k}^{*}$
are images of $G$ and $H$ respectively by the morphism of groups
$\pi_{k} : \diffh{}{n} \to D_{k}$,
% (truncation),
$H_{k}^{*}$ is a finite
index subgroup of $G_{k}^{*}$ for any $k \in {\mathbb N}$.
%The groups $\{ \phi_{k} : \phi \in G\} / \{\eta_{k} : \eta \in H\}$ are images of
%$G/H$ by  morphisms of groups and then finite.

Suppose $H$ is pro-algebraic. Then $H_{k}^{*}$ is algebraic for any
$k \in {\mathbb N}$ by Proposition \ref{pro:char}.
Hence $G_{k}^{*}$ is algebraic for any $k \in {\mathbb N}$.
In order to prove that $G$ is pro-algebraic, it suffices to show that $G$ is closed
in the Krull topology. Since $G$ is the union of finitely many
left cosets of $H$ and they
%classes in $G/H$ and they
are all closed in the Krull topology, we deduce that $G$ is closed in the Krull topology.

Suppose $G$ is pro-algebraic. Then $G_{0}$ is pro-algebraic by Remark \ref{rem:ccia}.
Moreover $G_{0}$
is contained in $H$ by Lemma \ref{lem:fiscg}. Since $G_{0}$ is a finite index subgroup of
$G$ and then of $H$, $H$ is pro-algebraic by the first part of the proof.
\end{proof}
\begin{lem}
\label{lem:lficl}
Let $H$  be a finite index normal subgroup of a subgroup $G$ of $\diffh{}{n}$.
Then $\overline{H}$ is a finite index normal subgroup of $\overline{G}$.
Moreover $H_{k}$ is a finite index normal subgroup of $G_{k}$ for any
$k \in {\mathbb N}$.
\end{lem}
\begin{proof}
%The property $\phi H \phi^{-1} = H$ implies $\phi \overline{H} \phi^{-1} = \overline{H}$
%for any $\phi \in G$.
Consider $\phi_{1}, \ldots, \phi_{m} \in G$ such that
$G/H = \{ \phi_{1} H, \ldots, \phi_{m} H \}$. We define the group
$J = \langle \overline{H} , \phi_{1}, \ldots, \phi_{m} \rangle$,
it satisfies $J \subset \overline{G}$.
Let us show that $J= \overline{G}$ and that $\overline{H}$ is a finite index
normal subgroup of $J$.

Since $\overline{H}$ is a normal subgroup of $J$ by Lemma \ref{lem:clnin},
every element $\psi$ of $J$ is of the form
\[ \psi = \phi_{i_{1}}^{\pm 1} \circ \ldots \circ  \phi_{i_{l}}^{\pm 1} \circ h  \]
where $h \in \overline{H}$. The choice of $\phi_{1}, \ldots, \phi_{m}$ implies
the existence of $1 \leq j \leq m$ and $h' \in H$ such that
$\psi = \phi_{j} \circ (h' \circ h)$. In particular the natural map
$G/H \to J/\overline{H}$ is surjective and hence $\overline{H}$ is a finite
index normal subgroup of $J$.
The group $J$ is pro-algebraic by Lemma \ref{lem:profis}.
Since $G \subset J \subset \overline{G}$, we deduce
$J= \overline{G}$ and $\overline{H}$ is a finite index normal subgroup
of $\overline{G}$. Since
$G_k$ and $H_k$ are images of $\overline{G}$ and $\overline{H}$
respectively by the
%natural
morphism
%of groups
$\pi_{k} : \overline{G} \to G_{k}$,
%$G_{k}= \pi_{k}(\overline{G})$ and \marginpar{$\pi_{k}$}
%$H_{k}= \pi_{k}(\overline{H})$,
$H_{k}$ is a finite index normal subgroup of
$G_{k}$ for any $k \in {\mathbb N}$.
%
%
%Analogously the natural map
%\[ G/H \to \{ \eta_{k} : \eta \in J\} / H_{k} \]
%is surjective for any $k \in {\mathbb N}$ and hence $H_{k}$ is a finite index normal
%subgroup of $ \{ \eta_{k} : \eta \in J\}$ for any $k \in {\mathbb N}$.
%Since $H_{k}$ is algebraic, $ \{ \eta_{k} : \eta \in J\}$ is algebraic for any $k \in {\mathbb N}$.
%As a consequence $\overline{J}$ is the closure of $J$ in the Krull topology.
%Since $\overline{H}$ is a finite index normal subgroup of $J$ that is closed in the
%Krull topology, $J$ is closed in the Krull topology. Thus $J$ is pro-algebraic.
%Notice that $G \subset J \subset \overline{G}$ implies
%$\overline{G} = \overline{J} = J$.
%Since  $G_{k}/H_{k}$ is equal to $J_{k}/H_{k}$, it is a finite group for any $k \in {\mathbb N}$.
\end{proof}
\subsection{Algebraic properties of the Zariski-closure}
The groups $G$ and $\overline{G}$ share many algebraic properties.
\begin{defi}
Let $G$ be a group. Given $f,g \in G$ we define by $[f,g]=fgf^{-1}g^{-1}$ the commutator of
$f$ and $g$.

Given subgroups $H, L$ of $G$ we define
$[H,L] = \langle [h,l] : h \in H,  \ l \in L \rangle$ as the subgroup generated by
the commutators of elements of $H$ and elements of $L$.
\end{defi}
\begin{defi}
Let $G$ be a group. By induction we define the subgroups
\[ G^{(0)}=G, \ G^{(1)} = [G^{(0)}, G^{(0)}], \ldots, \ G^{(\ell +1)} =  [G^{(\ell)}, G^{(\ell)}], \ldots \]
of the derived series of $G$.
We say that $G^{(\ell)}$ is the $\ell$-th derived group of $G$.
We use sometimes the notation $G'$ instead of $G^{(1)}$ for the derived group of $G$.

We say that $G$ is {\it solvable} if there exists $\ell \in {\mathbb N} \cup \{0\}$
such that $G^{(\ell)}=\{1\}$. We define the {\it derived length} of $G$ as the minimum
$\ell \in {\mathbb N} \cup \{0\}$ with such a property.
\end{defi}
\begin{defi}
\label{def:nil}
Let $G$ be a group. By induction we define the subgroups
\[ {\mathcal C}^{0} G =G, \ {\mathcal C}^{1} G = [{\mathcal C}^{0} G, G], \ldots,
\ {\mathcal C}^{\ell +1} G =  [{\mathcal C}^{\ell}, G], \ldots \]
of the descending central series of $G$.
We say that $G$ is {\it nilpotent} if there exists $\ell \in {\mathbb N} \cup \{0\}$
such that $ {\mathcal C}^{\ell} G =\{1\}$. We define the {\it nilpotence class} of $G$ as the minimum
$\ell \in {\mathbb N} \cup \{0\}$ with such a property.
\end{defi}
\begin{lem}
\label{lem:pap}
Let $G$ be a subgroup of $\diffh{}{n}$.  We have
\begin{itemize}
\item $G$ is abelian if and only if $\overline{G}$ is abelian.
\item $G$ is solvable if and only if $\overline{G}$ is solvable.
\item $G$ is nilpotent if and only if $\overline{G}$ is nilpotent.
\end{itemize}
\end{lem}
The two first properties were proved in \cite[Lemma 1]{JR:arxivdl}.
The proof of the last one is completely analogous.
All these properties are a consequence of a simple principle: the derived length
(resp. the nilpotence class) does not change when we take the Zariski-closure of
a matrix group.
\begin{defi}
\label{def:vir}
Let $G$ be a group and $P$ a group property. We say that $G$ is
{\it virtually} $P$ if there exists a finite index subgroup $H$ of $G$ that satisfies $P$.
\end{defi}
\begin{rem}
If the property $P$ is subgroup-closed (for instance solubility or nilpotence) then
we can suppose that the group $H$ is a finite index normal subgroup of $G$
(cf. \cite[1.6.9, p. 36]{Robinson}).
\end{rem}
\begin{lem}
\label{lem:fiscc}
Let $G$ be a subgroup of $\diffh{}{n}$.
Then the following properties are equivalent:
\begin{enumerate}
\item $G$ is virtually nilpotent (resp. solvable).
\item $\overline{G}$ is virtually nilpotent (resp. solvable).
\item $\overline{G}_{0}$ is nilpotent (resp. solvable).
\end{enumerate}
%Then $G$ is virtually nilpotent (resp. solvable) if and only if
%$\overline{G}_{0}$ is nilpotent (resp. solvable).
\end{lem}
\begin{proof}
Let us show the result in the virtually nilpotent case. The other case
is analogous.

Let us show $(1) \implies (2)$. 
Let $H$ be a finite index normal nilpotent subgroup of $G$.
Then $\overline{H}$ is a finite index normal nilpotent subgroup of
$\overline{G}$ by Lemmas \ref{lem:lficl} and \ref{lem:pap}.
Hence $\overline{G}$ is virtually nilpotent. 

Let us prove $(2) \implies (3)$. 
There exists a finite index normal nilpotent subgroup 
$J$ of  $\overline{G}$. The group $J$ contains 
$\overline{G}_{0}$ by Lemma \ref{lem:fiscg} and thus 
$\overline{G}_{0}$ is nilpotent. 

Let us see that $(3)$ implies $(1)$.
Since $\overline{G}_{0}$ is a finite index normal subgroup of 
$\overline{G}$ by Remark \ref{rem:ccia}, $\overline{G}$ is 
virtually nilpotent. 
The group $G$ is a subgroup of
$\overline{G}$ and hence also virtually nilpotent.
%
%Moreover $\overline{H}$ contains $\overline{G}_{0}$ by
%Lemma \ref{lem:fiscg} and thus $\overline{G}_{0}$ is nilpotent.
%Suppose that $\overline{G}_{0}$ is nilpotent. Since $\overline{G}_{0}$
%is a finite index normal subgroup of $\overline{G}$ by Remark \ref{rem:ccia},
%the latter group is virtually nilpotent. 
\end{proof}
\subsection{Jordan decomposition of formal diffeomorphisms}
Let us consider the multiplicative Jordan decomposition of formal diffeomorphisms
in commuting semisimple (or equivalently diagonalizable) and unipotent parts.
It was constructed by Martinet in \cite{MarJ}.
The analogous decomposition for algebraic matrix groups is called
Jordan-Chevalley decomposition since Chevalley showed
\begin{teo}[Chevalley, {cf. \cite[section I.4.4, p. 83]{Borel}}]
\label{teo:chevj}
Let $H$ be an algebraic matrix group. Then the
semisimple and unipotent parts of the elements of $H$ also belong to $H$.
\end{teo}
We will see that Chevalley's theorem also guarantees the closedness of
the Jordan decomposition for pro-algebraic subgroups of $\diffh{}{n}$.
\begin{defi}
We say that $\phi \in \diffh{}{n}$ is {\it semisimple} if $\phi_{k}$ is
semisimple   for any $k \in {\mathbb N}$.
\end{defi}
\begin{rem}
\label{rem:u}
By definition $\phi \in \diffh{}{n}$ is unipotent if and only if $\phi_{1}$ is unipotent.
It is not difficult to show that $\phi$ is unipotent if and only if $\phi_{k}$ is unipotent
for any $k \in {\mathbb N}$ (cf.  \cite[Proposition 3.12]{rib:cimpa}).
\end{rem}
\begin{rem}
It is well-known that $\phi$ is semisimple if and only if $\phi$ is formally conjugated
to a linear diagonal map (cf.
\cite[Lemma 2.9]{rib-embedding}  \cite[Proposition 3.13]{rib:cimpa}).
\end{rem}
Given $\phi \in \diffh{}{n}$ we consider the multiplicative Jordan decomposition of
$\phi_{k}$ for $k \in {\mathbb N}$.
The semisimple and unipotent parts $\phi_{k,s}$
and $\phi_{k,u}$ of $\phi_k$ belong to the algebraic group $D_{k}$ by
Chevalley's Theorem \ref{teo:chevj}.
Moreover since $\pi_{k,l}(\phi_{k,s})$ is semisimple,
$\pi_{k,l}(\phi_{k,u})$ is unipotent and
\[ \phi_{l} = \pi_{k,l}(\phi_{k}) = \pi_{k,l}(\phi_{k,s}) \pi_{k,l}(\phi_{k,u})
=   \pi_{k,l}(\phi_{k,u}) \pi_{k,l}(\phi_{k,s}), \]
we deduce
$\pi_{k,l}(\phi_{k,s}) = \phi_{l,s}$ and $\pi_{k,l}(\phi_{k,u}) = \phi_{l,u}$
for any $k \geq l \geq 1$ by uniqueness of the Jordan-Chevalley decomposition.
Hence $(\phi_{k,s})_{k \in {\mathbb N}}$ and $(\phi_{k,u})_{k \in {\mathbb N}}$
define elements $\phi_{s}$ and $\phi_{u}$ in $\diffh{}{n}= \varprojlim D_{k}$ respectively.
This leads to the next well-known result.
\begin{pro}
Let $\phi \in \diffh{}{n}$. There exist unique elements $\phi_{s}$ and $\phi_{u}$ in
$\diffh{}{n}$ such that
\[ \phi = \phi_{s} \circ \phi_{u} = \phi_u \circ \phi_s, \]
$\phi_s$ is semisimple and $\phi_u$ is unipotent.
\end{pro}
\begin{defi}
Let $G$ be a subgroup of $\diffh{}{n}$.
We say that $G$ is {\it splittable} if $\phi_s$, $\phi_u \in G$ for any $\phi \in G$.
\end{defi}
Chevalley's Theorem \ref{teo:chevj} implies
\begin{pro}
\label{pro:chevj}
Let $G$ be a subgroup of $\diffh{}{n}$.  Then $\overline{G}$ is splittable.
\end{pro}
Analogously there exists an additive Jordan decomposition for formal vector fields.
\begin{pro}
Let $X \in \hat{\mathfrak X} ({\mathbb C}^{n},0)$. There exist
unique elements $X_{s}$ and $X_{N}$ in
$\hat{\mathfrak X} ({\mathbb C}^{n},0)$ such that
\[ X = X_{s} + X_{N}  \ \ \mathrm{and} \ \ [X_s, X_N]=0, \]
$X_s$ is semisimple (i.e. formally conjugated to a linear diagonal vector field)
and $X_N$ is nilpotent.
\end{pro}
\begin{rem}
It is clear that if $X=X_{s} + X_{N}$ is the additive Jordan decomposition of
$X \in \hat{\mathfrak X} ({\mathbb C}^{n},0)$ then
$\mathrm{exp}(X) = \mathrm{exp}(X_s) \circ \mathrm{exp}(X_N)$ is the multiplicative
Jordan decomposition of $\mathrm{exp}(X)$.
\end{rem}
\begin{rem}
\label{rem:closs}
Given a semisimple $\phi \in \diffh{}{n}$ it is easy to calculate $\overline{\langle \phi \rangle}$.
Indeed $\phi$ is of the form $\phi(z_1,\ldots,z_n)= (\lambda_1 z_1, \ldots, \lambda_n z_n)$
in some formal system of coordinates. In such coordinates $\overline{\langle \phi \rangle}$
coincides with the Zariski-closure of the group
$\langle \mathrm{diag} (\lambda_1, \ldots, \lambda_n) \rangle$ in $\mathrm{GL}(n,{\mathbb C})$.
It can be described in terms of characters. We have
\[ \overline{\langle \phi \rangle} =
\{ \mathrm{diag} (\mu_1, \ldots, \mu_n) :
(\mu_1, \ldots, \mu_n) \in
\cap_{\underline{a} \in {\mathbb Z}^{n}, \ (\lambda_1, \ldots, \lambda_n) \in \ker (\chi_{\underline{a}})}
\ker (\chi_{\underline{a}}) \} \]
where given $\underline{a} = (a_1,\ldots,a_n) \in {\mathbb Z}^{n}$
we consider the character $\chi_{\underline{a}}: ({\mathbb C}^{*})^{n} \to {\mathbb C}^{*}$
defined by $\chi_{\underline{a}}(\mu_1, \ldots, \mu_{n})= \mu_{1}^{a_{1}}  \ldots \mu_{n}^{a_n}$.
\end{rem}
Let us calculate the algebraic closure of a cyclic subgroup of $\diffh{}{n}$.
\begin{defi}
Given a complex manifold $M$ we denote its dimension  by $\dim M$.
Given a complex vector space $V$ we denote its dimension by $\dim V$.
\end{defi}
\begin{lem}
\label{lem:clcg}
Let $\phi \in \diffh{}{n}$. Then $\overline{\langle \phi \rangle}$ is an abelian group
that is isomorphic to the product
$\overline{\langle \phi_s \rangle} \times \overline{\langle \phi_u \rangle}$.
Moreover ${\langle \phi \rangle}_k$ is isomorphic to the product
$\langle \phi_{s} \rangle_{k} \times \langle \phi_{u} \rangle_{k}$ and
$\dim {\langle \phi \rangle}_k = \dim \langle \phi_{s} \rangle_{k} + \dim \langle \phi_{u} \rangle_{k}$
for any $k \in {\mathbb N}$.
\end{lem}
\begin{proof}
The group $\overline{\langle \phi \rangle}$ is abelian by Lemma \ref{lem:pap}.
By Proposition \ref{pro:chevj} the formal diffeomorphisms
$\phi_s$ and $\phi_u$ belong to $\overline{\langle \phi \rangle}$
and then $\overline{\langle \phi \rangle}$ contains
the group
$H := \langle \overline{\langle \phi_s \rangle} , \overline{\langle \phi_u \rangle} \rangle$.
We claim $\overline{\langle \phi \rangle} = H$, it suffices to show that $H$
is pro-algebraic.

Remark
\ref{rem:closs} implies that
$\overline{\langle \phi_s \rangle}$ consists of semisimple elements
and is closed in the Krull topology. Hence $\langle \phi_{s} \rangle_{k} $ is
composed of semisimple elements for any $k \in {\mathbb N}$.
Analogously $\overline{\langle \phi_u \rangle}$ is contained in $\diffh{u}{n}$,
is closed in the Krull topology by Remark  \ref{rem:closu} and
$\langle \phi_{u} \rangle_{k} $ consists of unipotent elements for any $k \in {\mathbb N}$.

The group  $H_{k}^{*}$ is the image of the morphism
\[ \begin{array}{ccc}
\langle \phi_{s} \rangle_{k} \times \langle \phi_{u} \rangle_{k} & \stackrel{\iota}{\to} & D_{k}  \\
(\alpha, \beta) & \mapsto & \alpha \beta
\end{array} \]
of algebraic groups.
Hence  $H_{k}^{*}$ is an algebraic subgroup of $D_{k}$
for any $k \in {\mathbb N}$ (cf. \cite[2.1 (f), p. 57]{Borel}).
Since
$\langle {\langle \phi_s \rangle}_k , {\langle \phi_u \rangle}_k \rangle$ is abelian, the uniqueness
of the Jordan decomposition implies $\iota$ is injective.
Thus $H_k$ is isomorphic to
$\langle \phi_{s} \rangle_{k} \times \langle \phi_{u} \rangle_{k}$ and satisfies
\[ \dim H_k = \dim \langle \phi_{s} \rangle_{k} + \dim \langle \phi_{u} \rangle_{k} \]
for any $k \in {\mathbb N}$.
%In the previous formula we are considering the dimensions of the corresponding
%complex manifolds.
In order to conclude the proof
it suffices to show that
$H$ is closed in the Krull topology by Proposition \ref{pro:char}.

Every element $\eta$ of $H$ is of the form $\psi \circ \rho$ where
$\psi \in  \overline{\langle \phi_s \rangle}$ and
$\rho \in  \overline{\langle \phi_s \rangle}$.
Since $H$ is abelian, $\psi$ is semisimple and $\rho$ is unipotent,  $\psi \circ \rho$
is the multiplicative Jordan decomposition of $\eta$.
Moreover $H$ is isomorphic to 
$\overline{\langle \phi_s \rangle} \times \overline{\langle \phi_u \rangle}$
by uniqueness of the Jordan-Chevalley decomposition.
Thus $\overline{\langle \phi_s \rangle}$
(resp. $\overline{\langle \phi_u \rangle}$) is the set
of semisimple (resp. unipotent) elements of $H$.
Given a sequence $(\eta_{k})_{k \in {\mathbb N}}$ in $H$ that converges in the Krull topology,
the sequences $(\eta_{k,s})_{k \in {\mathbb N}}$ and $(\eta_{k,u})_{k \in {\mathbb N}}$
are contained in $\overline{\langle \phi_s \rangle}$ and $\overline{\langle \phi_u \rangle}$
respectively and
both converge in the Krull topology. Since
$\overline{\langle \phi_s \rangle}$ and $\overline{\langle \phi_u \rangle}$
are closed in the Krull topology so is $H$.
\end{proof}
\section{Finite dimensional groups of formal diffeomorphisms}
\label{sec:fdgfd}
Our main goal is characterizing the groups $G$ of local diffeomorphisms that can be embedded
in finite dimensional Lie groups. We approach this problem from a canonical point of view.
Indeed we provide an invariant $\dim G$ of $G$ such that $\dim G < \infty$ implies that
the Zariski-closure $\overline{G}$ of $G$ is algebraic or more precisely that
the map $\pi_{k} : \overline{G} \to G_{k}$ is an isomorphism of groups for some
$k \in {\mathbb N}$.
In such a case $\pi_k^{-1}: G_k \to \overline{G}$ can be interpreted as an algebraic
morphism  and $\overline{G}$ as a matrix algebraic group (in particular as
a complex Lie group with finitely many connected components).
On the other hand we will see that a Lie subgroup of $\diffh{}{n}$ with finitely many connected components
(cf. Definition \ref{def:lie}) is finite dimensional (Proposition \ref{pro:lieialg} and Lemma \ref{lem:fis}).
%
%A priori it could be possible for a subgroup $G$ of $\diffh{}{n}$ to satisfy
%$\dim G = \infty$ and being the image by a morphism of a real Lie group with finitely many
%connected components. We will see that it never happens and thus our canonical
%approach encloses the results by
%Seigal-Yakovenko  and Binyamini \cite{Seigal-Yakovenko:ldi, Binyamini:finite}.

There are other advantages of working with the Zariski-closure of a group of local
diffeomorphisms. For instance given a normal subgroup $H$ of a group
$G \subset \diffh{}{n}$ we can naturally define whether
the extension is finite dimensional. A straightforward consequence of the definition is that
$G$ is finite dimensional if and only if it is a tower of finite dimensional extensions of
the trivial group. Hence it is natural to identify finite dimensional extensions. The following
kind of extensions are finite dimensional:
\begin{enumerate}
\item $H$ is a finite index subgroup of $G$. \\
\item $G/H$ is a finitely generated abelian group. \\
\item $G/H$ is a  connected Lie group.
\end{enumerate}
These items are generalizations of the cases treated in \cite{Seigal-Yakovenko:ldi, Binyamini:finite}
in the context of extensions of groups.
Thus a natural strategy to show $\dim G < \infty$ for a subgroup $G$ of $\diffh{}{n}$ is
decomposing it as a tower
of extensions of the types (1), (2) and (3) of the trivial group.
This method allows to generalize Theorem
\ref{teo:mainc} to much bigger classes of groups.
%For instance Theorem \ref{teo:mainc}
%holds for finitely generated nilpotent groups since these groups are towers of cyclic groups.
\subsection{Dimensional setting}
The first step of our program is defining the dimension of an extension of subgroups of
$\diffh{}{n}$.
\begin{lem}
\label{lem:dc}
Let $G$ be a subgroup of $\diffh{}{n}$.
Consider a subgroup $H$ of $G$. Then
$\dim G_{k} - \dim H_{k} \leq \dim G_{k+1} - \dim H_{k+1}$
for any $k \in {\mathbb N}$.
\end{lem}
\begin{proof}
Let ${\mathfrak g}_{k}$ and ${\mathfrak h}_{k}$
be the Lie algebra of $G_k$ and $H_{k}$ respectively for $k \in {\mathbb N}$.
Since $\pi_{k+1,k}:G_{k+1} \to G_k$ is surjective and we are working in
characteristic $0$, we obtain $(d \pi_{k+1,k})_{Id} : {\mathfrak g}_{k+1} \to {\mathfrak g}_{k}$
is surjective for any $k \in {\mathbb N}$ (cf. \cite[Chapter II.7, p. 105]{Borel}).
Moreover $(d \pi_{k+1,k})_{Id} ( {\mathfrak h}_{k+1})$ is equal to $ {\mathfrak h}_{k}$
for $k \in {\mathbb N}$. Therefore the linear map
$(d \pi_{k+1,k})_{Id} : {\mathfrak g}_{k+1}/{\mathfrak h}_{k+1} \to {\mathfrak g}_{k}/{\mathfrak h}_{k}$
is surjective. Since
\[   \dim {\mathfrak g}_{k} - \dim {\mathfrak h}_{k} \leq
\dim {\mathfrak g}_{k+1} - \dim {\mathfrak h}_{k+1}   \]
we deduce
$\dim G_{k} - \dim H_{k} \leq \dim G_{k+1} - \dim H_{k+1}$
for any $k \in {\mathbb N}$.
\end{proof}
Since $(\dim G_{k} - \dim H_{k})_{k \geq 1}$ is increasing we can define the codimension of
$H$ in $G$.
\begin{defi}
\label{def:dim}
Consider a subgroup $H$ of a subgroup $G$ of $\diffh{}{n}$.
We define
$\dim G/H  \in {\mathbb Z}_{\geq 0} \cup \{ \infty \}$ as
\[ \dim G/H = \lim_{k \to \infty} \dim G_{k} - \dim H_{k}. \]
We say that $G/H$ is {\it finite dimensional} or that $H$ has finite codimension in
$G$ if $\dim G/H < \infty$.
Notice that we can define $\dim G$ for any subgroup $G$ of $\diffh{}{n}$
by considering $H= \{Id\}$.
\end{defi}
\begin{rem}
Notice that the definition does not distinguish between a subgroup of $\diffh{}{n}$
and its Zariski-closure. More precisely, if $H$ is a subgroup of a group
$G \subset \diffh{}{n}$, we have $\dim \overline{G}= \dim G$ and
$\dim \overline{G}/\overline{H} = \dim G/H$.
\end{rem}
The following result is an immediate consequence of the definition.
It will be useful to know whether or not a subgroup of $\diffh{}{n}$
is finite dimensional in practical applications since it allows to divide the problem
in simpler ones.
\begin{pro}
\label{pro:elems}
Consider a sequence
$G^{1} \subset G^{2} \subset \ldots \subset G^{m}$ of subgroups
of $\diffh{}{n}$. Then we obtain
\[ \dim G^{m}/G^{1} = \dim G^{m}/G^{m-1} + \ldots + \dim G^{3}/G^{2} + \dim G^{2}/G^{1} . \]
In particular
$G^{m}/G^{1}$ is finite dimensional if and only if
$G^{j+1}/G^{j}$ is finite dimensional for any $1 \leq j < m$.
\end{pro}
Next proposition provides several characterizations of finite dimensional extensions.
\begin{pro}
\label{pro:elem}
Let $G$ be a subgroup of $\diffh{}{n}$.
Let $H$ be a subgroup of $G$.
The following properties are equivalent:
\begin{enumerate}
\item There exists $k_{0} \in {\mathbb N}$ such that $\phi \in \overline{G}$ and
$\phi_{k_{0}} \in H_{k_{0}}$ imply $\phi \in \overline{H}$.
\item There exists $k_{0} \in {\mathbb N}$ such that the map
$\hat{\pi}_{k_{0}}: \overline{G}/\overline{H} \to G_{k_{0}}/H_{k_{0}}$,
induced by $\pi_{k_{0}}: \overline{G} \to G_{k_{0}}$, is injective.
\item There exists $k_{0} \in {\mathbb N}$ such that the map
$\hat{\pi}_{k+1,k}: G_{k+1}/H_{k+1} \to G_{k}/H_{k}$ (induced by ${\pi}_{k+1,k}$) is
injective for any $k \geq k_{0}$.
\item $G/H$ is finite dimensional.
\end{enumerate}
\end{pro}
%
%\begin{rem}
%The map $\hat{\pi}_{k+1,k}: G_{k+1}/H_{k+1} \to G_{k}/H_{k}$ is the map induced by
% ${\pi}_{k+1,k}: G_{k+1} \to G_{k}$.
%\end{rem}
%
\begin{rem}
\label{rem:injiso}
Since we are not supposing that $H$ is normal we consider left cosets.
The proposition can be strengthened if $H$ is normal. Then $\overline{H}$
is normal in $\overline{G}$ and $H_{k}$ is normal in $G_{k}$ for
any $k \in {\mathbb N}$ by Lemma \ref{lem:clnin}.
Notice that
$G_{k}/H_k$  is an algebraic group for any $k \in {\mathbb N}$
(cf. \cite[section II.6.8, p. 98]{Borel}).
Hence $\hat{\pi}_{k+1,k}$ is a morphism of
algebraic groups.
Since $\hat{\pi}_{k}$ is always surjective for $k \in {\mathbb N}$,
condition (2) is equivalent to
$\hat{\pi}_{k_{0}}$ being an isomorphisms of groups
from $\overline{G}/\overline{H}$ onto the algebraic matrix group
$G_{k_{0}}/H_{k_{0}}$.
Condition (3) is equivalent
to $\hat{\pi}_{k+1,k}$ being a bijective morphism of algebraic matrix groups
and then an isomorphism of algebraic groups
(cf. \cite[Theorem 6, Chapter 3.1.4]{alg.lie:Onis-Vinb}).
\end{rem}
%\begin{rem}
%Regarding the third item
%notice that if $\pi_{k+1,k}: G_{k+1} \to G_{k}$ is injective then it is a bijective morphism
%of algebraic groups and hence an isomorphism of algebraic groups.
%\end{rem}
\begin{proof}
%Since the map $\hat{\pi}_{k_{0}}$ is surjective, 
The first two properties are clearly equivalent.

Let us show $(2) \implies (3)$.  We have
$\hat{\pi}_{k_{0}} = \hat{\pi}_{k,k_{0}} \circ \hat{\pi}_{k+1, k} \circ \hat{\pi}_{k+1}$
for $k \geq k_{0}$. Since the maps  $\hat{\pi}_{k,k_{0}}$,  $\hat{\pi}_{k+1, k}$, $\hat{\pi}_{k+1}$
are surjective and $\hat{\pi}_{k_{0}}$ is injective,
$\hat{\pi}_{k,k_{0}}$,  $\hat{\pi}_{k+1, k}$, $\hat{\pi}_{k+1}$
are also injective for any $k \geq k_{0}$.
%Hence $\pi_{k+1,k}$ is a bijective morphism of groups and then an isomorphism
%for any $k \geq k_{0}$.

Let us show $(3) \implies (2)$. The map $\hat{\pi}_{k_{0}}$ is equal to
$\hat{\pi}_{k,k_{0}} \circ \hat{\pi}_{k}$ for $k \geq k_{0}$.
Since $\hat{\pi}_{k,k_{0}}=  \hat{\pi}_{k_{0}+1,k_{0}} \circ \ldots  \circ \hat{\pi}_{k,k-1}$
is a composition of injective maps by hypothesis, $\hat{\pi}_{k,k_{0}}$ is injective for $k \geq k_{0}$.
Given left cosets $\phi \overline{H}$ and $\eta \overline{H}$ such that
$\hat{\pi}_{k_{0}}(\phi \overline{H}) = \hat{\pi}_{k_{0}}(\eta \overline{H})$ we
obtain
$\hat{\pi}_{k}(\phi \overline{H}) = \hat{\pi}_{k}(\eta \overline{H})$
for any $k \geq k_{0}$.
We deduce $(\eta^{-1} \phi)_{k} \in H_{k}$ for any $k \geq k_{0}$.
In particular we have $\eta^{-1} \phi \in \overline{H}$ and hence
$\phi \overline{H} = \eta \overline{H}$. Thus $\hat{\pi}_{k_{0}}$ is
injective.

Let us show $(3) \implies (4)$. The map $\hat{\pi}_{k+1,k}$ is an isomorphism
of algebraic manifolds for any $k \geq k_{0}$ by the universal mapping property
of quotient morphisms (cf. \cite[Chapter II.6]{Borel}).
%(cf. \cite[Theorem 6.13, Chapter II.6]{Borel})
We deduce 
\[ \dim G_{k+1} - \dim H_{k+1} = \dim G_{k} - \dim H_{k} \]
for any $k \geq k_{0}$.

%It is an immediate consequence of Remark
%\ref{rem:injiso}.

Let us show $(4) \implies (3)$.
Let ${\mathfrak g}_{k}$ and ${\mathfrak h}_{k}$
be the Lie algebra of $G_k$ and $H_{k}$ respectively for $k \in {\mathbb N}$.
There exists $k_{0} \in {\mathbb N}$ such that
\[ \dim G_{k} - \dim H_{k} = \dim G_{k_{0}} - \dim H_{k_{0}} \]
for any $k \geq k_{0}$.
The linear map
$(d \pi_{k+1,k})_{Id} : {\mathfrak g}_{k+1} \to {\mathfrak g}_{k}$
is surjective for any $k \in {\mathbb N}$ by the proof of Lemma \ref{lem:dc}.
Since $(d \pi_{k+1,k})_{Id}({\mathfrak h}_{k+1})={\mathfrak h}_{k}$ for any
$k \in {\mathbb N}$ the linear map
$(d \hat{\pi}_{k+1,k})_{Id} : {\mathfrak g}_{k+1}/{\mathfrak h}_{k+1} \to
{\mathfrak g}_{k}/{\mathfrak h}_{k}$ is well-defined and surjective for any $k \in {\mathbb N}$.
Since both complex vector spaces ${\mathfrak g}_{k+1}/{\mathfrak h}_{k+1}$ and
${\mathfrak g}_{k}/{\mathfrak h}_{k}$ have the same dimension, the map
$(d \hat{\pi}_{k+1,k})_{Id}$ is a linear isomorphism for any $k \geq k_{0}$.

Fix $k \geq k_{0}$. Let us show that $\hat{\pi}_{k+1,k}$ is injective.
Let $A \in G_{k+1}$ such that $\hat{\pi}_{k+1,k} (A H_{k+1}) = H_{k}$.
We have $\pi_{k+1,k}(A) \in H_{k}$. The restriction
$(\pi_{k+1,k})_{|H_{k+1}} : H_{k+1} \to H_{k}$ is surjective by Remark \ref{rem:pik},
hence there exists $B \in H_{k+1}$ such that $\pi_{k+1,k}(A)=\pi_{k+1,k}(B)$.
We obtain $\pi_{k+1,k}(B^{-1} A) = Id$.
There exists  $\phi \in \overline{G}$ such that $\phi_{k+1} = B^{-1} A$
since $\pi_{k+1}: \overline{G} \to G_{k+1}$ is surjective by Remark \ref{rem:pik2}.
Since $\phi_{k} \equiv Id$ the linear part $D_{0} \phi$ of $\phi$ at $0$ is equal to $Id$
and thus $\log \phi$ belongs to the Lie algebra
${\mathfrak g}$ of $\overline{G}$ (Remark \ref{rem:closu}) and satisfies
$(\log \phi)_{k} \equiv 0$. The property $(d \pi_{k+1,k})_{Id} ((\log \phi)_{k+1})=0$ and
the injective nature of $(d \hat{\pi}_{k+1,k})_{Id}$ imply $(\log \phi)_{k+1} \in {\mathfrak h}_{k+1}$.
Since $B^{-1} A = \mathrm{exp} ((\log \phi)_{k+1} )$ we obtain $ B^{-1} A \in H_{k+1}$
and then $A \in H_{k+1}$.
Hence $\pi_{k+1,k}$ is injective for any $k \geq k_{0}$.
\end{proof}
\begin{rem}
\label{rem:linjdim}
Notice that the proof of $(3) \Leftrightarrow (4)$
in Proposition \ref{pro:elem} implies that $\hat{\pi}_{k+1,k}: G_{k+1}/H_{k+1} \to G_{k}/H_{k}$ is
bijective if and only if
$\dim G_{k} - \dim H_{k} = \dim G_{k+1} - \dim H_{k+1}$.
Such a property validates our point of view
since the dimension determines an extension of the form $G_k/H_k$
for $k \in {\mathbb N}$ modulo isomorphism.

Moreover if $(3)$ holds then $\dim G/H= \dim G_{k_{0}} - \dim H_{k_{0}}$.
We have $\dim G/H= \dim G_{k_{0}} - \dim H_{k_{0}}$ if $(2)$ holds by the proof
of $(2) \implies (3)$.
% Two extensions $G_{k}/H_{k}$ and $G_{l}/H_{k}$ are isomorphic
%if and only if their dimensions coincide.
\end{rem}
\begin{rem}
\label{rem:hotlot}
Let $G$ be a finite dimensional subgroup of $\diffh{}{n}$.
There exists $k_{0} \in {\mathbb N}$ such that $\pi_{k,k_{0}}: G_{k} \to G_{k_{0}}$ is
an isomorphism of algebraic matrix groups for any $k \geq k_{0}$.
%by Proposition \ref{pro:elem}.
Consider the Taylor series expansion
\[ \phi (z_{1}, \ldots, z_{n})= (\sum_{i_{1}+\ldots+i_{n} \geq 1} a_{i_{1} \ldots i_{n}}^{1}
z_{1}^{i_{1}} \ldots z_{n}^{i_{n}}, \ldots, \sum_{i_{1}+\ldots+i_{n} \geq 1} a_{i_{1} \ldots i_{n}}^{n}
z_{1}^{i_{1}} \ldots z_{n}^{i_{n}} ) \]
of $\phi \in G$.
Given $(i_{1}, \ldots, i_{n};j)$ a multi-index such that $i_{1}+ \ldots +i_{n} > k_{0}$ and
$1 \leq j \leq n$
the function $a_{i_{1} \ldots i_{n}}^{j}: G \to {\mathbb C}$ belongs to the affine
coordinate ring ${\mathbb C}[G_{k_{0}}]$ of $G_{k_{0}}$.
%Then there exists a function $P_{i_{1}\ldots i_{n}}^{j}$ in the affine
%coordinate ring ${\mathbb C}[G_{k_{0}}]$ of $G_{k_{0}}$ such that the function
%$a_{i_{1} \ldots i_{n}}^{j}: G \to {\mathbb C}$ belongs to ${\mathbb C}[G_{k_{0}}]$.
In other words every coefficient
in the Taylor expansion, of an element of $G$,
of degree greater than $k_{0}$ is a regular function $P_{i_{1} \ldots i_{n}}^{j}$
on the coefficients
of degree less or equal than $k_{0}$.

Reciprocally suppose that  $a_{i_{1} \ldots i_{n}}^{j}: G \to {\mathbb C}$ is a polynomial function of 
the coefficients of degree less or equal than $k_0$ for any 
multi-index $(i_{1}, \ldots, i_{n};j)$ such that $i_{1}+ \ldots +i_{n} > k_{0}$ and
$1 \leq j \leq n$ (meaning that there exists $P_{i_{1} \ldots i_{n}}^{j} \in {\mathbb C}[D_{k_{0}}]$
such that 
\[ a_{i_{1} \ldots i_{n}}^{j}(\phi) = 
P_{i_{1} \ldots i_{n}}^{j} ((a_{l_{1} \ldots l_{n}}^{m}(\phi))_{l_1 +\cdots + l_n \leq k_{0}, \ 1 \leq m \leq n}) \]
for any $\phi \in G$). Since all these equations hold true in the Zariski-closure $G_{i_{1} + \cdots + i_{n}}$
of $G_{i_{1} + \cdots + i_{n}}^{*}$, we have that 
$\pi_{k,k_{0}}^{*} : {\mathbb C}[G_{k_{0}}] \to {\mathbb C}[G_{k}]$ is an isomorphism of 
${\mathbb C}$-algebras and in particular $\pi_{k,k_{0}}: G_{k} \to G_{k_{0}}$ 
is an isomorphism of algebraic groups for any $k \geq k_0$.
Thus $G$ is finite dimensional.
\end{rem}
Let us relate the finite dimension property with a much simpler one, namely the finite determination
property.
\begin{defi}
\label{def:findet}
Let $G$ be a subgroup of $\diffh{}{n}$. We say that $G$ has the
{\it finite determination property} if there exists $k \in {\mathbb N}$ such
that $\phi \in G$ and $\phi_{k} \equiv Id$ imply $\phi \equiv Id$.
\end{defi}
\begin{rem}
\label{rem:fdc}
Let us compare the finite determination and the finite dimension properties. On the one hand
a subgroup $G$ of $\diffh{}{n}$ has the
finite determination property if there exists  $k \in {\mathbb N}$ such
that the projection $\pi_k : G \to D_{k}$ is injective. On the other hand $G$ is finite
dimensional if there exists  $k \in {\mathbb N}$ such
that $\pi_k :\overline{G} \to D_{k}$ is injective.
\end{rem}
\begin{rem}
\label{rem:fdefd}
Notice that a subgroup $G$ of $\diffh{}{n}$ is finite dimensional if and only if
$\overline{G}$ has the finite determination property.
\end{rem}
\begin{rem}
\label{rem:fdnifd}
Every finite dimensional group has finite determination but in general the reciprocal does
not hold true.
We define
\[ \phi(j)(x,y)= (x , y + d_{j} x^{2} + x^{j+2}) \in \diff{}{2} \]
for $j \in {\mathbb N}$.
Suppose that the subset $S:=\{ d_{1}, d_{2}, \ldots \}$ of
%complex numbers
${\mathbb C}$ is linearly independent
over ${\mathbb Q}$.
We have $\log \phi(j) = (d_{j} x^{2} + x^{j+2}) \partial / \partial{y}$ for $j \in {\mathbb N}$.
We denote by $G$ the group generated by $\{\phi(1), \phi(2), \ldots \}$.
It is an abelian group.
Moreover since $S$ is linearly independent over ${\mathbb Q}$, the property
$\phi \neq Id$ implies $\phi_{2} \neq Id$ for any $\phi \in G$.
In particular $G$ has the finite determination property.

By choice the complex Lie algebra generated by $\{ \log \phi(1), \log \phi(2), \ldots \}$ is infinite
dimensional as a complex vector space.
This implies that $\overline{G}$ contains non-trivial elements whose
order of contact with the identity is arbitrarily high or in other words that the
map $\pi_{k}: \overline{G} \to D_{k}$ is not injective for any $k \in {\mathbb N}$.
Hence $\overline{G}$ does not
have the finite determination property and $G$ is not finite dimensional.

An example of finite determination group that is not finite dimensional does not exist
in dimension $1$ (Proposition \ref{pro:fd1fd}).
\end{rem}
\begin{rem}
We define
\[ \phi(j)(x,y)= (x , y + d_{j} x^{2} + z^{j+2},z) \in \diff{}{3} \]
for $j \in {\mathbb N}$ where the subset
 $S:=\{ d_{1}, d_{2}, \ldots \}$ of
%complex numbers
${\mathbb C}$ is linearly independent
over ${\mathbb Q}$. Let $G$ be the group generated by $\{\phi(1), \phi(2), \ldots \}$.
Analogously as in the previous example $G$ is finitely determined but it is not finite dimensional.
We have 
\[ (\phi(j)^{-1} \{ x=y=0\} , \{y=0\}) = \dim \frac{{\mathcal O}_{3}}{ (x, y + d_{j} x^{2} + z^{j+2}, y)}  
% \dim {\mathcal O}_{3}/ (x, y ,  z^{j+2}) 
= j+2 \]
for any $j \in {\mathbb N}$. As a consequence finite determination does not suffice to guarantee 
the uniform intersection property (cf. Theorem \ref{teo:main}).
\end{rem}
\begin{defi}
We say that a subgroup $G$ of $\diffh{}{n}$ is {\it algebraic} if
$G$ is pro-algebraic and $\dim G < \infty$.
\end{defi}
\begin{rem}
\label{rem:eae}
An algebraic subgroup $G$ of $\diffh{}{n}$ is the image by an algebraic monomorphism of
an algebraic matrix group.
Given $k_{0} \in {\mathbb N}$ such that
${\pi}_{k_{0}}: \overline{G} \to G_{k_{0}}$ is injective, the map
${\pi}_{k_{0}}^{-1}: G_{k_{0}} \to G$ is an isomorphism of groups (Remark \ref{rem:injiso}).
Moreover, it is algebraic in every jet space since
${\pi}_{k} \circ {\pi}_{k_{0}}^{-1}: G_{k_{0}} \to G_{k}$ is the inverse of
the algebraic isomorphism $\pi_{k,k_{0}}: G_{k} \to G_{k_{0}}$ for any
$k \geq k_{0}$ (Remark \ref{rem:injiso}).
\end{rem}
The characterization of pro-algebraic groups given by Proposition \ref{pro:char}
provides a characterization of algebraic subgroups of $\diffh{}{n}$.
\begin{lem}
\label{lem:chalg}
Let $G$ be a subgroup of $\diffh{}{n}$.
Then $G$ is algebraic if and only if $G_{k}^{*}$ is algebraic
for any $k \in {\mathbb N}$ and the sequence $(\dim G_{k}^{*})_{k \geq 1}$
is bounded.
\end{lem}
\begin{proof}
The group $G$ is pro-algebraic if and only if $G_{k}^{*}$ is algebraic
for any $k \in {\mathbb N}$ and $G$ is closed in the Krull topology
by Proposition \ref{pro:char}.

The sufficient condition is obvious. Let us show the necessary condition.
%By construction we have $G_{k} = \{  \phi_{k} : \phi \in G \}$ for any $k \in {\mathbb N}$.
It suffices to show that $G$ is closed in the Krull topology.
There exists $k_{0} \in {\mathbb N}$ such that $\pi_{k+1,k}:G_{k+1} \to G_{k}$ is injective for
any $k \geq k_{0}$ by Remark \ref{rem:linjdim}.
We deduce that the map $\pi_{k_{0}}: G \to G_{k_{0}}$ is injective.
Consider a sequence $(\eta_{m})_{m \geq 1}$ of elements of $G$ that converge in the Krull topology.
Then there exists $m_{0} \in {\mathbb N}$ such that $(\eta_{m})_{k_{0}} \equiv (\eta_{m_{0}})_{k_{0}}$
if $m \geq m_{0}$. Therefore $\eta_{m} \equiv \eta_{m_{0}}$ for any $m \geq m_{0}$ and the
sequence converges to $\eta_{m_{0}} \in G$. We obtain that $G$ is closed in the Krull topology.
\end{proof}
%Let us introduce some elementary properties of finite dimensional groups of local
%diffeomorphisms.
Let us provide the first examples of finite dimensional groups. Indeed we will see
that cyclic groups and one-parameter groups are always finite dimensional.
\begin{pro}
\label{pro:cloif}
Let $\phi \in \diffh{}{n}$. We have $\dim \langle \phi \rangle \leq n$.
\end{pro}
\begin{proof}
We have
%${\langle \phi \rangle}_{k}
%= \langle {\langle \phi_{s} \rangle}_{k}, {\langle \phi_{u} \rangle}_{k} \rangle$
%by Lemma \ref{lem:clcg}.
%
%The semisimple and unipotent parts $\phi_{s}$ and $\phi_{u}$ of $\phi$ belong to
%$\overline{\langle \phi \rangle}$ by Chevalley's theorem. Moreover
%since $\overline{\langle \phi \rangle}$  is abelian,
%\[
%\begin{array}{ccc}
%{\langle \phi_{s} \rangle}_{k}  \times  {\langle \phi_{u} \rangle}_{k} & \to & {\langle \phi \rangle}_{k} \\
%(A, B) & \mapsto & A B
%\end{array}
%\]
%is a morphism of algebraic groups for any $k \in {\mathbb N}$.
%The image $\langle {\langle \phi_{s} \rangle}_{k}, {\langle \phi_{u} \rangle}_{k} \rangle$
%by such morphism is an algebraic group containing $\langle \phi_{k} \rangle$ and
%contained in ${\langle \phi \rangle}_{k}$.
%We obtain
%${\langle \phi \rangle}_{k}
%= \langle {\langle \phi_{s} \rangle}_{k}, {\langle \phi_{u} \rangle}_{k} \rangle$
%and then
\[ \dim {\langle \phi \rangle}_{k}  =
\dim  {\langle \phi_{s} \rangle}_{k} + \dim  {\langle \phi_{u} \rangle}_{k} \]
for any $k \in {\mathbb N}$ by Lemma \ref{lem:clcg}.
It suffices to show that $\langle \phi_{s} \rangle$ and $\langle \phi_{u} \rangle$
are finite dimensional.

Since $\phi_{s}$ is semisimple, we can suppose up to a formal change of coordinates
that
%it is of the form
%$\phi_{s}(z_{1},\ldots,z_{n})= (\lambda_{1} z_{1}, \ldots, \lambda_{n} z_{n})$.
%The group
$\overline{\langle \phi_{s} \rangle}$ is contained in the group of
diagonal matrices (Remark \ref{rem:closs}). We deduce $\dim \langle \phi_{s} \rangle \leq n$.

Since $\phi_{u}$ is unipotent, we obtain
$\overline{\langle \phi_{u} \rangle}=
\{ \phi_{u}^{t} : t \in {\mathbb C} \}$
and in particular
${\langle \phi_{u} \rangle}_{k}= \{ \mathrm{exp} (t \log \phi_{u,k}) : t \in {\mathbb C} \}$
for any $k \in {\mathbb N}$ by Remark \ref{rem:closu}. We deduce
$\dim \langle \phi_{u} \rangle = 1$ if $\phi_{u} \not \equiv Id$ and
$\dim \langle Id \rangle = 0$.
We get
\[ \dim \langle \phi \rangle = \dim\langle \phi_{s} \rangle + \dim \langle \phi_{u} \rangle \leq n+1. \]
In order to show $\dim \langle \phi \rangle \leq n$
let us prove that $\dim \langle \phi_{s} \rangle =n$
implies $\phi_{u} \equiv Id$.
Indeed in such a case $\overline{\langle \phi_{s} \rangle}$ is the linear group of diagonal
trasformations. Every element of such group commutes with $\phi_{u}$
by Lemma \ref{lem:clcg} and hence $\phi_{u}$ is linear and diagonal.
Since $\phi_{u}$ is both semisimple and unipotent, it is equal to the identity map.
\end{proof}
The finite dimension of one-parameter groups can be obtained by reduction to the cyclic case.
More precisely, we will use that every one-parameter group $G$ of formal diffeomorphisms has
cyclic subgroups whose Zariski-closure coincides with $\overline{G}$.
\begin{pro}
\label{pro:fginfp}
Let $X \in \hat{\mathfrak X} \cn{n}$. Then there exists $t_0 \in {\mathbb R}$
%$\phi \in \{ \mathrm{exp} (t X) : t \in {\mathbb R} \}$
such that
$\{ \mathrm{exp} (t X) : t \in {\mathbb C} \} \subset \overline{\langle \mathrm{exp} (t_0 X) \rangle}$.
In particular we obtain 
\[ \dim \{ \mathrm{exp} (t X) : t \in {\mathbb C} \} \leq n. \]
\end{pro}
\begin{proof}
Consider the Jordan decomposition $X = X_{s} + X_{N}$ as
a sum of commuting formal vector fields such that $X_{s}$ is formally
diagonalizable and $X_{N}$ is nilpotent.
We can suppose $X_{s} = \sum_{k=1}^{n} \mu_{k} z_{k} \partial / \partial z_{k}$
where $\mu_{1}, \ldots, \mu_{n} \in {\mathbb C}$ up to a formal
change of coordinates.
We denote 
\[ D = \{ \underline{a}  \in {\mathbb Z}^{n} : \sum_{k=1}^{n} a_{k} \mu_{k} \neq 0 \} \] 
where $\underline{a} =(a_{1},\ldots,a_{n})$. 
Given $\underline{a} \in D$ we define
\[ C_{\underline{a}} =
\{ t \in {\mathbb C} : t  \sum_{k=1}^{n} a_{k} \mu_{k} \in 2 \pi i {\mathbb Z} \}; \]
it is a countable set. Consider an element
$t_{0}$ in the complementary of the countable set
$\cup_{ \underline{a} \in D}  C_{\underline{a}}$ in ${\mathbb R}^{*}$.
We define
\[ \eta  (z_{1},\ldots,z_{n})= \mathrm{exp} (t_{0} X_{s})=
(e^{t_{0} \mu_{1}} z_{1}, \ldots, e^{t_{0} \mu_{n}} z_{n}) \in \diffh{}{n} . \]
The group ${\mathcal C}$ of characters
$\chi_{\underline{a}}(w_{1},\ldots,w_{n}) = w_{1}^{a_{1}} \ldots w_{n}^{a_{n}}$
with $\underline{a} \in {\mathbb Z}^{n}$ defined by
\[ {\mathcal C} =
\{ \chi_{\underline{a}} : (e^{t_{0} \mu_{1}}, \ldots, e^{t_{0} \mu_{n}}) \in \ker (\chi_{\underline{a}}) \} \]
satisfies
${\mathcal C} = \{ \chi_{\underline{a}} : \sum_{k=1}^{n} a_{k} \mu_{k} =0 \}$ by our choice of
$t_{0}$. The group $\overline{\langle \eta \rangle}$
consists of the linear diagonal maps
$\mathrm{diag} (\lambda_1, \ldots, \lambda_n)$ such that
$\chi_{\underline{a}} (\lambda_1, \ldots, \lambda_n) = 1$ for any
$\chi_{\underline{a}} \in {\mathcal C}$.
Since $(e^{t \mu_{1}}, \ldots, e^{t \mu_{n}}) \in  \ker (\chi_{\underline{a}})$
for all $\chi_{\underline{a}} \in {\mathcal C}$
and $t \in {\mathbb C}$, the one parameter group
$\{ \mathrm{exp} (t X_{s}) : t \in {\mathbb C} \}$ is contained in
$\overline{\langle \eta \rangle}$.
We denote $\rho = \mathrm{exp} (t_{0} X_{N})$, it
satisfies 
\[ \overline{\langle \rho \rangle} = \{ \mathrm{exp} (t X_{N}) : t \in {\mathbb C}\} \]
by Remark \ref{rem:closu}.  We denote
$\phi = \mathrm{exp} (t_{0} X)= \eta \circ \rho$.
Since
$\overline{\langle \phi \rangle}$ contains
\[ \overline{\langle \phi_{s} \rangle} \cup \overline{\langle \phi_{u} \rangle} =
\overline{\langle \eta \rangle} \cup \overline{\langle \rho \rangle}, \]
it also contains $\{ \mathrm{exp} (t X) : t \in {\mathbb C} \}$. Hence
$\dim \{ \mathrm{exp} (t X) : t \in {\mathbb C} \} \leq n$ is a consequence of
Proposition \ref{pro:cloif}.
\end{proof}
The finite dimensional nature of a subgroup of $\diffh{}{n}$ is related to properties of finite
decomposition
of the elements of the group in terms of generators. The following results illustrates
how a finite writing property allows to decide whether or not $\dim G < \infty$ by solving
simpler problems.
\begin{pro}
\label{pro:fw}
Let $H_{1}, \ldots, H_{m}$ and $G$ be subgroups of $\diffh{}{n}$.
Suppose $G \subset H_{1} \ldots H_{m}$.
Then we have
\[ \dim G \leq \sum_{1 \leq j \leq m} \dim H_{j}. \]
Moreover given $1 \leq j \leq m$ such that $H_{j} \subset G$ we obtain
\[ \dim G/H_{j} \leq  \sum_{k \neq j} \dim H_{k}. \]
\end{pro}
We denote $H_{1} \ldots H_{m} =
\{ h_{1} \circ \ldots \circ h_{m}: h_{j} \in H_{j} \  \forall 1 \leq j \leq m\}$.
\begin{proof}
Fix $k \in {\mathbb N}$. Consider the map
\[
\begin{array}{ccccccccccc}
\tau_k & : & H_{m,k} & \times &  H_{m-1,k} & \times & \ldots & \times & H_{1,k}
&  \to & D_{k} \\
& & (B_{m} & , & B_{m-1} & , & \ldots & , & B_{1}) & \mapsto & B_{m} B_{m-1} \ldots B_{1}.
\end{array}
\]
The map $\tau_k$ is algebraic even if it is not in general a morphism of groups.
As a consequence $\mathrm{Im} (\tau_k)$ is a constructible set whose dimension is less or
equal than $\sum_{1 \leq j \leq m} \dim H_{j,k}$. The Zariski-closure of
$\mathrm{Im} (\tau_k)$ is an algebraic set containing $G_{k}^{*}$ and then $G_{k}$.
%Thus the Zariski-closure of $\mathrm{Im} (\tau_k)$ is equal to $G_k$.
We deduce $\dim G_{k} \leq \sum_{1 \leq j \leq m} \dim H_{j,k}$ for any $k \in {\mathbb N}$.
The results are a direct consequence of Definition \ref{def:dim}.
\end{proof}
\begin{teo}
\label{teo:fw}
Let $G$ be a subgroup of $\diffh{}{n}$.
Suppose there exist $\psi_{1}, \ldots, \psi_{l} \in \diffh{}{n}$,
$X_1, \ldots, X_m \in \hat{\mathfrak X} \cn{n}$ and $p \in {\mathbb N}$ such that
every $\phi \in G$ is of the form
$\phi = \phi_{1} \circ \ldots \circ \phi_{q}$ where
$q \leq p$ and
$\phi_{r} \in \cup_{j=1}^{l} \langle \psi_{j} \rangle \cup \cup_{k=1}^{m}
\{ \mathrm{exp} (t X_{j}) : t \in {\mathbb C} \}$ for any $1 \leq r \leq q$.
Then $G$ is finite dimensional.
\end{teo}
\begin{proof}
The result is a straightforward consequence of Propositions \ref{pro:fw},
\ref{pro:cloif} and \ref{pro:fginfp}.
\end{proof}
\begin{rem}
\label{rem:fw}
The elements of a connected Lie group admit a decomposition in words of uniform length
whose letters are taken from the elements of  a finite set of cyclic and one-parameter subgroups.
This can be seen as a consequence of the Iwasawa-Malcev decomposition of
a connected Lie group (cf. \cite[Theorem 6]{Iwasawa:top} \cite{Malcev:large}).
Such a property generalizes inmediately to Lie groups with finitely
many connected components and then to algebraic matrix groups.
Since every algebraic group of formal diffeomorphisms is isomorphic
to an algebraic matrix group,  finite dimension can be interpreted
as a finite decomposition property.
%
%the hypothesis of Theorem \ref{teo:fw} holds for
%any finite dimensional subgroup of $\diffh{}{n}$. Therefore finite dimension can be interpreted
%as a finite decomposition property.
\end{rem}
\subsection{Extensions of groups of formal diffeomorphisms}
In this section we study extensions that are always finite dimensional.
First, we deal with  finite extensions.
%Finite dimensionality can be checked out in finite index subgroups.
\begin{lem}
\label{lem:fis}
Let $G$ be a subgroup of $\diffh{}{n}$. Consider a finite index subgroup $H$ of
$G$. Then we obtain $\dim G/H =0$.
\end{lem}
\begin{proof}
There exists a subgroup $J$ of $H$ such that $J$ is a finite index normal
subgroup of $G$. Since $J_{k}$ is a finite index normal subgroup of $G_{k}$
by Lemma \ref{lem:clnin}, we obtain $\dim J_{k} = \dim G_{k}$ for any $k \in {\mathbb N}$.
We deduce $\dim G/J=0$ and then $\dim G/H=0$ since $J \subset H$.
\end{proof}
\begin{cor}
\label{cor:fis}
Let $G$ be a subgroup of $\diffh{}{n}$.
Consider subgroups $H,K$ of $G$ such that $H \subset K$ and
$K$ is a a finite index subgroup of $G$. Then
$\dim G/H = \dim K/H$.
\end{cor}
\begin{proof}
We have $\dim G/H = \dim G/K + \dim K/H = \dim K/H$ by Lemma \ref{lem:fis}.
%There exists a subgroup $J$ of $K$ such that $J$ is a finite index normal
%subgroup of $G$. Since $J_{k}$ is a finite index normal subgroup of $G_{k}$
%by Lemma \ref{lem:clnin}, we obtain $\dim J_{k} = \dim G_{k}$ for any $k \in {mathbb N}$.
%We deduce $\dim G/J=0$ and then $\dim G/K=0$ since $J \subset K$.
%
%
%Since $J_{k}$ is a finite index normal subgroup of both $K_{k}$ and
%$G_{k}$ for any $k \in {\mathbb N}$ by Lemma \ref{lem:lficl}, we deduce
%$\dim K_{k} = \dim J_{k} = \dim G_{k}$ for any $k \in {\mathbb N}$.
%We obtain $\dim G/H = \dim K/H$ since
%$\dim G_{k} - \dim H_{k} = \dim K_{k} - \dim H_{k}$ for any $k \in {\mathbb N}$.
\end{proof}
Next, we will consider finitely generated abelian extensions.
First let us discuss the finite generation hypothesis.
A positive dimensional connected Lie group is not finitely generated since it is not countable.
Anyway, it is finitely generated by a finite number of one-parameter groups
whose infinitesimal generators
are the elements of a basis of the Lie algebra. This idea inspires an
alternative definition of finitely generated subgroup of $\diffh{}{n}$   
in which generators can be elements of the group
or one-parameter flows.
\begin{defi}
Let $H$ be a subgroup of a subgroup $G$ of $\diffh{}{n}$.
We say that $G$ is {\it finitely generated in the extended sense} over $H$
if there exist elements $\phi_{1},\ldots,\phi_{l} \in G$
and formal vector fields $X_{1}, \ldots, X_{m} \in \hat{\mathfrak X} \cn{n}$ such
that
\[ G = \langle H, \phi_{1}, \ldots, \phi_{l},
\cup_{j=1}^{m} \{ \mathrm{exp} (t X_{j}) : t \in {\mathbb R} \} \rangle. \]
We say that $G/H$ is finitely generated in the extended sense if $H$ is normal
in $G$. If $G$ is of the form $\langle H, \phi_{1}, \ldots, \phi_{l} \rangle$ we say that
$G$ is finitely generated over $H$.
\end{defi}
We are interested in calculating the dimension of subgroups of $\diffh{}{n}$.
In this context the
new definition of finitely generated group
can be reduced to the usual one via the next lemma.
\begin{lem}
\label{lem:fginf}
Let $H$ be a subgroup of a subgroup $G$ of $\diffh{}{n}$.
Suppose that $G$ is finitely generated over $H$ in the extended sense.
Then there exists a subgroup $G_{+}$ of $\diffh{}{n}$ such that
$H \subset G_{+} \subset G$,
$\overline{G}_{+} = \overline{G}$ and
$G_{+}$ is finitely generated over $H$.
\end{lem}
\begin{proof}
Suppose $G = \langle H, \phi_{1}, \ldots, \phi_{l},
\cup_{j=1}^{m} \{ \mathrm{exp} (t X_{j}) : t \in {\mathbb R} \} \rangle$.
Given $1 \leq j \leq m$ there exists $t_{j} \in {\mathbb R}$ such that
$\psi_{j} :=  \mathrm{exp} (t_{j} X_{j})$ satisfies
$\{ \mathrm{exp} (t X_{j}) : t \in {\mathbb C}\} \subset \overline{\langle \psi_{j} \rangle}$
by Proposition \ref{pro:fginfp}.  We define
\[ G_{+} = \langle H, \phi_{1}, \ldots, \phi_{l}, \psi_{1}, \ldots, \psi_{m} \rangle . \]
It is clear that $H \subset G_{+} \subset G$. The choice of $\psi_j$ for $1 \leq j \leq m$
implies $G \subset \overline{G}_{+}$.
Since $G_{+} \subset G$, we obtain
$\overline{G} = \overline{G}_{+}$.
\end{proof}
\begin{pro}
\label{pro:fgaifd}
Let $H$ be a normal subgroup of a subgroup $G$ of $\diffh{}{n}$.
Suppose $G/H$ is abelian and $G/H$ is finitely generated in the
extended sense.
Then $G/H$ is finite dimensional.
\end{pro}
\begin{proof}
%There exists a subgroup $G_{+}$ of $G$ such that
%$H \subset G_{+}$, $\overline{G}_{+} = \overline{G}$
%and $G_{+} / H$ is finitely generated.
%Since $\overline{G}_{+} = \overline{G}$, we obtain
%$\dim G_{+}/H = \dim G/H$.
%Moreover $G_{+}/ H$ is a subgroup of
%$G / H$ and as a consequence abelian.
%Then up to replace
%$G$ with $G_{+}$ we
%can suppose that $G/H$ is finitely generated.
We have
$G = \langle H, \phi_{1}, \ldots, \phi_{l},
\cup_{j=1}^{m} \{ \mathrm{exp} (t X_{j}) : t \in {\mathbb R} \} \rangle$.
We denote $H_{j}= \langle \phi_{j} \rangle$ for $1 \leq j \leq l$
and $J_{k}=\{ \mathrm{exp} (t X_{k}) : t \in {\mathbb R} \}$ for $1 \leq k \leq m$.
Since $H$ is normal in $G$ and $G/H$ is abelian, we obtain
$G=H_1 \ldots H_{l} J_1 \ldots J_m H$. This implies
\[ \dim G/H \leq \sum_{1 \leq j \leq l} \dim H_{j} + \sum_{1 \leq k \leq m} \dim J_{k}  \leq (l+m) n \]
by Propositions \ref{pro:fw}, \ref{pro:cloif} and \ref{pro:fginfp}.
\end{proof}
\begin{rem}
Proposition \ref{pro:fgaifd} implies in particular
that Theorem \ref{teo:main} can be applied to
a finitely generated
(in the extended sense) abelian subgroup of $\diffh{}{n}$.
So Seigal-Yakovenko's Theorem \ref{teo:sy} can be understood as a consequence of the finite dimension of
finitely generated abelian subgroups of formal diffeomorphisms.
%The conclusion of Theorem \ref{teo:main}
%in such a case was originally proved by Seigal and Yakovenko in
%\cite{Seigal-Yakovenko:ldi}.
\end{rem}
Next, let us consider the third type of extensions, namely extensions that are
real connected Lie groups. Of course
the first task is finding a proper definition of Lie group for an extension since it is
not clear a priori.
\begin{defi}
\label{def:lie}
Let $H$ be a normal subgroup of a subgroup $G$ of $\diffh{}{n}$.
We say that $G/H$ is a (connected) Lie group
if there exists a (connected) Lie group $L$ and
a surjective morphism of groups $\sigma: L \to G/H$ such that the map
$\sigma_{k}: L \to D_{k}/H_k$
induced by $\sigma$
is a morphism of differentiable manifolds for any $k \in {\mathbb N}$.
\end{defi}
Notice that $D_{k}/H_{k}$ is a smooth algebraic manifold.
The map $\pi_{k}:G \to D_{k}$ induces a map $\hat{\pi}_{k}': G/H \to D_{k}/H_{k}$.
The map $\sigma_{k}$ is equal to $\hat{\pi}_{k}' \circ \sigma$.
\begin{rem}
The previous definition of Lie group coincides with the usual one for a subgroup
$G$ of $\diffh{}{n}$, i.e. in the case $H=\{Id\}$ (cf. \cite{Binyamini:finite}).
\end{rem}
Let us see that connected Lie group extensions
are finite dimensional (Proposition \ref{pro:lieialge}). Proposition \ref{pro:lieialg}
is a corollary of
%both Corollary \ref{cor:fw} (cf. Remark \ref{rem:fw}) and
Proposition \ref{pro:lieialge}.
\begin{pro}
\label{pro:lieialg}
Let $G$ be a subgroup of $\diffh{}{n}$. Suppose that $G$ is a connected Lie group.
Then $G$ is finite dimensional.
%
%Let $\sigma: L \to \diffh{}{n}$ be a morphism of groups such that
%the natural morphism $\sigma_{k}: L \to D_{k}$
%is a morphism of Lie groups for any
%$k \in {\mathbb N}$. Suppose that the real Lie group $L$ has a finite number of
%connected components. Then $\sigma (L)$ has finite dimension.
\end{pro}
\begin{pro}
\label{pro:lieialge}
Let $H$ be a normal subgroup of a subgroup $G$ of $\diffh{}{n}$.
%Suppose that $\{ \phi_{k}: \phi \in H \}$ is algebraic for any $k \in {\mathbb N}$.
%Furthermore suppose that $G/H$ is a Lie group with a finite number of connected components,
Suppose that $G/H$ is a connected Lie group.
%meaning that there exists a Lie group $L$ with a finite number of connected components and
%a surjective morphism of groups $\sigma: L \to G/H$ such that the map
%$\sigma_{k}: L \to D_{k}/H_k$
%induced by $\sigma$
%is a morphism of differentiable manifolds for any $k \in {\mathbb N}$.
Then $G/H$ is finite dimensional.
\end{pro}
\begin{proof}
Let us explain the idea of the proof. Suppose $H=\{Id\}$ and $L$
(cf. Definition \ref{def:lie}) is a connected complex Lie subgroup of
$\mathrm{GL}(n,{\mathbb C})$.
Then the derived group $L'$ is  algebraic
(cf. \cite[Chapter 3.3.3]{alg.lie:Onis-Vinb}).
We can think of $L$ as an algebraic-by-finitely generated commutative  group
since $L / L'$ is abelian and finitely generated in the extended sense. Since all these
extensions are finite dimensional in a natural way, hence the image of $L$ is finite dimensional.

Consider the real Lie group $L$ and the surjective morphism of groups $\sigma: L \to G/H$
provided by Definition \ref{def:lie}.
Fix $k \in {\mathbb N}$.
Let ${\mathfrak g}_{k}$ and ${\mathfrak h}_{k}$
be the Lie algebras of $G_k$ and $H_{k}$.
The set $G_{k}/H_{k}$  is an algebraic group for any $k \in {\mathbb N}$
(cf. \cite[section II.6.8, p. 98]{Borel}).
Let ${\mathfrak g}$ be the Lie algebra of $L$.
%We denote by $L_k$ the Lie algebra
%of $D_k$ for $k \in {\mathbb N}$.
Consider the map
%$d \sigma_{k} : {\mathfrak g} \to \hat{\mathfrak X} \cn{n}$ induced by $\sigma_{k}$
$(d \sigma_{k})_{Id} : {\mathfrak g} \to {\mathfrak g}_{k}/{\mathfrak h}_k$ induced by $\sigma_{k}$
for $k \in {\mathbb N}$. We define by
$\tilde{\mathfrak g} ={\mathfrak g} \otimes_{\mathbb R} {\mathbb C}$
the complexified  of the Lie algebra ${\mathfrak g}$.
%We denote by $d \overline{\sigma}_{k} : \tilde{\mathfrak g} \to \hat{\mathfrak X} \cn{n}$
We denote by $(d \overline{\sigma}_{k})_{Id} : \tilde{\mathfrak g} \to {\mathfrak g}_{k}/{\mathfrak h}_k$
the morphism of complex Lie algebras induced by $(d \sigma_{k})_{Id}$.
Let $\tilde{L}$ be a connected simply connected complex Lie group whose Lie algebra is
equal to $\tilde{\mathfrak g}$.
Then there exists a unique morphism
$\tilde{\sigma}_k: \tilde{L} \to G_{k}/H_k$ of complex Lie groups
such that $(d \tilde{\sigma}_{k})_{Id} = (d \overline{\sigma}_{k})_{Id}$
for any $k \in {\mathbb N}$ (cf. \cite[Chapter 1.2.8]{alg.lie:Onis-Vinb}).
Notice that $\tilde{\sigma}_k (\tilde{L})$ contains $\sigma_k (L)$.
%Therefore
%up to replace $L$ with $\tilde{L}$ and $\sigma_k$ with $\tilde{\sigma}_k$ we can suppose that
%$L$ is a complex Lie group and $\sigma_{k}:L \to G_{k}/H_k$ is a morphism of complex Lie groups
%for any $k \in {\mathbb N}$.

We denote by $\rho_{k} : G_{k} \to G_{k}/H_{k}$ the morphism of algebraic groups given
by the projection.
Since $\tilde{\sigma}_{k}(\tilde{L})'$ is the derived group
of the connected complex Lie group of matrices $ \tilde{\sigma}_{k}(\tilde{L})$,
it is an algebraic subgroup of $G_{k}/H_{k}$. Hence
$\rho_{k}^{-1}( \tilde{\sigma}_{k}(\tilde{L})' )$ is an algebraic subgroup of $G_{k}$
such that
\[ \dim \rho_{k}^{-1}( \tilde{\sigma}_{k}(\tilde{L})' )
- \dim H_{k} \leq \dim \tilde{L}' \leq \dim  \tilde{L}
\leq  \dim_{\mathbb R} L  \]
where $\dim_{\mathbb R} L $ is the dimension of the real Lie group $L$.
Since ${\langle G', H \rangle}_{k}^{*}$ is contained in
$ \rho_{k}^{-1}( \tilde{\sigma}_{k}(\tilde{L}') )$ and the latter group is algebraic,
we obtain ${\langle G', H \rangle}_{k} \subset \rho_{k}^{-1}( \tilde{\sigma}_{k}(\tilde{L}') )$ and then
\[ \dim {\langle G', H \rangle}_{k} - \dim H_{k} \leq \dim_{\mathbb R} L   \]
for any $k \in {\mathbb N}$. Therefore
the extension $\langle G', H \rangle/H$ is finite dimensional.
%Since $\langle G', H \rangle/H$ is finite dimensional,
It suffices to show that $G/\langle G', H \rangle$ is finite
dimensional by Proposition \ref{pro:elems}.

Let $X_{1}, \ldots, X_{m}$ be a basis of  ${\mathfrak g}$ (as a real Lie algebra).
Fix $k \in {\mathbb N}$. Analogously as in the proof of Proposition \ref{pro:fginfp}
there exists a countable subset $A_{k}$ of ${\mathbb R}$ such that
the algebraic closure of $\langle \sigma_{k} (\mathrm{exp}(t X_{1})) \rangle$
in $G_k / H_k$ contains
the one-parameter group $\sigma_{k} \{ \mathrm{exp}(s X_{1}) : s \in {\mathbb C} \}$ for any
$t \in {\mathbb R}^{*} \setminus A_{k}$. We choose
$t \in {\mathbb R}^{*} \setminus \cup_{k \geq 1} A_{k}$
and a representative $\psi_{1} \in G$ of the class in $G/H$ defined by
$\sigma (\mathrm{exp}(t X_{1}))$. Analogously we define $\psi_{2}, \ldots, \psi_{m}$.
It is clear that the extension $\langle G',H, \psi_{1}, \ldots, \psi_{m} \rangle/ \langle G', H \rangle$
is abelian
and finitely generated. It suffices to show
$\overline{\langle G',H, \psi_{1}, \ldots, \psi_{m} \rangle} = \overline{G}$ by
Proposition \ref{pro:fgaifd}. We will prove
$\overline{\langle H, \psi_{1}, \ldots, \psi_{m} \rangle} = \overline{G}$.

Fix $k \in {\mathbb N}$.
We denote $J=\langle H, \psi_{1}, \ldots, \psi_{m} \rangle$.
% \[ C_{k}=\{ \phi_{k} : \phi \in \langle G',H, \psi_{1}, \ldots, \psi_{m} \rangle \}
%\ \mathrm{and} \ J_{k} = {\langle G',H, \psi_{1}, \ldots, \psi_{m} \rangle}_{k} . \]
The image of the Zariski-closure $J_{k}$
of $J_{k}^{*}$ by the morphism $\rho_{k}$
%of algebraic groups
is equal to the closure of $\langle J_{k}^{*}, H_k \rangle/H_{k}$ in $G_{k}/H_{k}$.
The Zariski-closure of $\langle \rho_{k} (\psi_{j,k}) \rangle$ contains
$\sigma_{k} (\{ \mathrm{exp} (s X_{j}) : s \in {\mathbb C} \})$ for any $1 \leq j \leq m$
by the choice of $\psi_{j}$.
Hence the closure of $\langle J_{k}^{*}, H_k \rangle/H_k$ in $G_{k}/H_{k}$ contains
$\langle G_{k}^{*}, H_k \rangle / H_k$.
We deduce that the closure of  $\rho_{k}(J_{k}^{*})$ in $G_{k}/H_{k}$
is equal to $G_k / H_k$.
%is equal to $G_k / H_k = \overline{\{ \phi_k : \phi \in G\}} / H_k$.
Thus $\rho_{k} (J_{k})$ is equal to $G_{k}/H_{k}$.
Since $J_{k}$ contains $H_{k}$,   we get $J_{k} = G_{k}$ for
any $k \in {\mathbb N}$. Hence we obtain
$\overline{\langle H, \psi_{1}, \ldots, \psi_{m} \rangle} = \overline{G}$.
\end{proof}
\begin{rem}
We can recover Binyamini's Theorem \ref{teo:b} in the context of the theory of
finite dimensional groups of formal diffeomorphisms by applying Theorem
\ref{teo:main},  Proposition \ref{pro:lieialg} and Lemma \ref{lem:fis}.

Binyamini proves that the matrix coefficients of $G$ belong to a noetherian ring
by using the Iwasawa-Malcev decomposition of a connected Lie group
(cf. \cite[Theorem 6]{Iwasawa:top}). Let us remark that
it is possible to use such decomposition to show Proposition  \ref{pro:lieialg}
(cf. Remark \ref{rem:fw}). Our choice of
proof is intended to stress the efficacy of the approach through towers of extensions.
Indeed given a connected Lie group $G \subset \diffh{}{n}$, we showed that its derived group $G'$
is finite dimensional by using classical results of Lie group theory. This allowed us
to reduce the problem to treat the finitely generated (in the
extended sense) and abelian extension $G/G'$.
\end{rem}
\begin{rem}
In the proof of Proposition \ref{pro:lieialge} it suffices to consider a linear independent
subset $\{X_{1}, \ldots, X_{m}\}$ of ${\mathfrak g}$ whose image in
${\mathfrak g}/{\mathfrak g}'$ is a basis.  As a consequence
%if $L$ is connected then
we obtain
\[ \dim \sigma (L) \leq \dim \tilde{L}' + (\dim \tilde{L} - \dim  \tilde{L}' ) n =
\dim  \tilde{L} + (\dim \tilde{L} - \dim  \tilde{L}' ) (n-1)  .\]
This formula reminds the formula in Theorem 4 of \cite{Binyamini:finite} in which the role
of a maximal connected compact subgroup  of $L$ is replaced with the derived group.
\end{rem}
\begin{rem}
Subgroups of Lie groups (with finitely many connected components)
of formal diffeomorphisms are always subgroups of algebraic
groups of formal diffeomorphisms by Proposition \ref{pro:lieialg} and Lemma \ref{lem:fis}.
 Given a subgroup $G$ of $\diffh{}{n}$ the existence 
of an embedding of $G$ in a Lie group of formal diffeomorphisms
implies that $G$ is contained in the image by an algebraic monomorphism of
an algebraic matrix group (Remark \ref{rem:eae}).
\end{rem}
The following theorem about finite determination properties
of Lie subgroups of $\diffh{}{n}$  is due to Baouendi et al.
\begin{teo}[{\cite[Proposition 5.1]{Baou-Roth-Win-Zai}}]
\label{teo:Baou}
Let $G$ be a subgroup of $\diffh{}{n}$. Suppose that $G$ is a Lie group
with a finite number of connected components. Then $G$ has the finite determination
property.
\end{teo}
We include a proof since it is an extremely easy application of our techniques.
\begin{proof}
The group $G$ is finite dimensional by Proposition  \ref{pro:lieialg} and Lemma \ref{lem:fis}.
Hence it has the finite determination property.
\end{proof}
In general the implication finite determination $\implies$ finite dimension does not
hold for subgroups $G$ of $\diffh{}{n}$ (cf. Remark \ref{rem:fdnifd}).
Anyway, it is interesting to explore in which conditions it is true since
the former property is much easier to verify.
Next, we see that the implication is satisfied, even for extensions, under
a property of algebraic closedness for cyclic subgroups.
\begin{teo}
\label{teo:charfdifd}
Let $H$ be a normal subgroup of a subgroup $G$ of $\diffh{}{n}$.
Suppose that the map $\hat{\pi}_{k_{0}}: \overline{G}/\overline{H} \to G_{k_{0}}/H_{k_{0}}$
satisfies $(\hat{\pi}_{k_{0}})_{|\langle G, \overline{H} \rangle / \overline{H}}$
is injective for some $k_{0} \in {\mathbb N}$.
Furthermore suppose  $\overline{\langle \phi \rangle} \subset \langle G, \overline{H} \rangle$
for any $\phi \in G$. Then $G/H$  is finite dimensional.
\end{teo}
%The condition $\overline{\langle \phi \rangle} \subset G$ is easy to verify if
%we know the Jordan-Chevalley decomposition of the elements of $G$.
\begin{proof}
Let $T$ be the subgroup of $\langle G, \overline{H} \rangle$ generated by
$\cup_{\phi \in G} \overline{\langle \phi \rangle}_{0}$. The property 
$\overline{\langle \phi \rangle} \subset \langle G, \overline{H} \rangle$ implies that 
the group $T_{k}^{*} = \langle \cup_{\phi \in G} {\langle \phi \rangle}_{k,0} \rangle$
is contained in $\langle G_{k}^{*}, H_{k} \rangle$. Moreover it is a (connected) algebraic group
since it is generated by a family of connected algebraic matrix groups
(Theorem \ref{teo:chevc}) for any $k \in {\mathbb N}$.
Theorem \ref{teo:chevc} also implies that $\langle T_{k}^{*}, H_{k,0} \rangle$ is algebraic.
Since $H_{k}$ is normal in $G_{k}$, we deduce that
$\langle T_{k}^{*}, H_{k,0} \rangle$  is a finite index subgroup of
$\langle T_{k}^{*}, H_{k} \rangle$ and in particular $\langle T_{k}^{*}, H_{k} \rangle$
is algebraic for any $k \in {\mathbb N}$.
The group $G_{k}^{*}$ normalizes the algebraic group $T_{k}^{*}$.
Thus $T_{k}^{*}$ is a normal subgroup of $G_{k}$.
As a consequence $\langle T_{k}^{*}, H_{k} \rangle$ is a subgroup of
$\langle G_{k}^{*}, H_{k} \rangle$ that is normal in $G_{k}$
for any $k \in {\mathbb N}$.

The group $\overline{\langle \phi \rangle}_{0}$ is a finite index normal subgroup of
$\overline{\langle \phi \rangle}$ and ${\langle \phi \rangle}_{k,0}$ is a finite index
normal subgroup of  ${\langle \phi \rangle}_{k}$ for all $\phi \in G$ and $k \in {\mathbb N}$.
We obtain that all elements of the subgroup
$\langle G_{k}^{*}, H_{k} \rangle / \langle T_{k}^{*}, H_{k} \rangle$
of the algebraic matrix group $G_{k}  / \langle T_{k}^{*}, H_{k} \rangle$
have finite order.
Suppose that the Zariski-closure of a group of matrices of elements of finite order consists
only of semisimple elements; we will prove this later on (Lemma \ref{lem:foiss}).
Since $G_{k}  / \langle T_{k}^{*}, H_{k} \rangle$ is the Zariski-closure of
$\langle G_{k}^{*}, H_{k} \rangle / \langle T_{k}^{*}, H_{k} \rangle$,
the group $G_{k}  / \langle T_{k}^{*}, H_{k} \rangle$ consists of semisimple elements.

Let us show that $\hat{\pi}_{k,k_{0}} :  G_{k}/H_{k} \to G_{k_{0}}/H_{k_{0}}$ is injective
for any $k \geq k_{0}$. Such property implies that $G/H$ is finite dimensional by
Proposition \ref{pro:elem}.
The hypothesis implies that
$(\hat{\pi}_{k,k_{0}})_{|\langle G_{k}^{*}, H_{k} \rangle / H_{k}}$ is injective.
Let $\phi_{k} H_{k}$ be an element of the kernel of $\hat{\pi}_{k,k_{0}}$ where
$\phi \in \overline{G}$. Since $\phi_{k_{0}} \in H_{k_{0}}$, there exists $\eta \in \overline{H}$
such that $\phi_{k_{0}} \equiv \eta_{k_{0}}$ and in particular $(\phi \circ \eta^{-1})_{k_{0}} \equiv Id$.
The formal diffeomorphism $\phi \circ \eta^{-1}$ is unipotent and so is
$(\phi \circ \eta^{-1})_{k}$.
Thus the class of $(\phi \circ \eta^{-1})_{k}$ in $G_{k}  / \langle T_{k}^{*}, H_{k} \rangle$
is unipotent. Since it is also semisimple by the previous discussion, we obtain
$(\phi \circ \eta^{-1})_{k} \in  \langle T_{k}^{*}, H_{k} \rangle$ and then
$\phi_{k} \in  \langle T_{k}^{*}, H_{k} \rangle \subset \langle G_{k}^{*}, H_{k} \rangle$.
Since $(\hat{\pi}_{k,k_{0}})_{|\langle G_{k}^{*}, H_{k} \rangle / H_{k}}$ is injective,
we obtain $\phi_{k} \in H_{k}$. In particular
$\hat{\pi}_{k,k_{0}} :  G_{k}/H_{k} \to G_{k_{0}}/H_{k_{0}}$ is injective for any $k \geq k_{0}$.
\end{proof}
\begin{cor}
\label{cor:charfdifd}
Let $G$ be a subgroup of $\diffh{}{n}$ that has the finite determination property.
Suppose that $\overline{\langle \phi \rangle}$ is contained in $G$ for any $\phi \in G$.
Then $G$ is finite dimensional.
\end{cor}
Let us remark, regarding Corollary \ref{cor:charfdifd}, that
the condition $\overline{\langle \phi \rangle} \subset G$ is easy to verify if
we know the Jordan-Chevalley decomposition of the elements of $G$.
\begin{lem}
\label{lem:foiss}
Let $G$ be a subgroup of $\mathrm{GL}(n,{\mathbb C})$ such that all its elements have finite order.
Then all elements of $\overline{G}$ are semisimple.
\end{lem}
\begin{proof}
The Tits alternative \cite{Tits}
implies that either $G$ is virtually solvable or it contains a non-abelian free group.
Since clearly the second option is impossible, $G$ is virtually solvable.
Hence $\overline{G}_{0}$ is solvable. Since it is also connected, the group
$\overline{G}_{0}$ is upper triangular up to a linear change of coordinates by Lie-Kolchin's
theorem (cf. \cite[section 17.6, p. 113]{Humphreys})

We denote $H = G \cap \overline{G}_{0}$; it is a finite index normal subgroup of $G$.
The derived group $H'$ of $H$ consists of unipotent
upper triangular matrices. They are also semisimple by hypothesis and as a consequence
$H'$ is the trivial group and
$H$ is abelian. Moreover $H$ consists of semisimple elements; hence
$H$ and then
$\overline{H}$ are diagonalizable. Since $H$ is a finite index normal subgroup of $G$,
$\overline{H}$ is a finite index normal subgroup of $\overline{G}$.
An element $A$ of $\overline{G}$ satisfies $A^{k} \in \overline{H}$ for some
$k \in {\mathbb N}$. Since $A^{k}$ is semisimple, $A$ is semisimple for any $A \in \overline{G}$.
\end{proof}
\section{Families of finite dimensional groups}
\label{sec:ffdg}
The finite dimensional subgroups of $\diffh{}{}$ can be characterized, they are the solvable groups.
\begin{pro}
\label{pro:fd1fd}
Let $G$ be a subgroup of $\diffh{}{}$. Then the following conditions are equivalent:
\begin{enumerate}
\item $G$ is solvable.
\item $G$ has the finite determination property.
\item $G$ is finite dimensional.
\end{enumerate}
\end{pro}
\begin{proof}
The items $(1)$ and $(2)$ are equivalent
(cf. \cite[Th\'{e}or\`{e}me 1.4.1]{Loray5b}).
%
%Given $\phi \in \diffh{}{} \setminus \{Id\}$ we define its fixed point index by
%$\min \{ k \in {\mathbb N}: \phi_{k} \not \equiv Id \}$.
%The definition makes sense since its coincide with the usual definition of
%fixed point index of $\phi$ at $0$ if $\phi \in \diff{}{}$.
%
%
%It is well-known that $G$ is non-solvable if and only if it contains two
%elements $\phi, \eta$ with different order of contact
%with the identity such $D_{0} \phi \equiv D_{0} \eta \equiv Id$.
%In such a case $G$ contains elements with arbitrarily high
%fixed point index.
%
%
%Clearly a solvable subgroup $G$ of $\diffh{}{}$ satisfies the finite determination property.
%
%Let us show $(2) \implies (1)$.
%The map $\pi_{k} : G \to D_{k}$ is injective for some $k \in {\mathbb N}$
%by hypothesis. Every element $\phi$ of the derived group $G'$ satisfies
%$D_{0} \phi \equiv Id$.
%The group $\pi_{k}(G')$ is a matrix group of unipotent elements and hence nilpotent
%by Kolchin's theorem. We deduce that $G'$ is nilpotent and then $G$ is solvable.

Let us show $(1) \implies (3)$. The group $\overline{G}$ is solvable by
Lemma \ref{lem:pap}. Since $(1) \implies (2)$, $\overline{G}$ has the finite
determination property and then $G$ is finite dimensional.

Let us prove $(3) \implies (1)$. Since $G$ is finite dimensional, $\overline{G}$
has the finite determination property. The property $(2) \implies (1)$
implies that $\overline{G}$ and then $G$ are solvable.
\end{proof}
\begin{rem}
Any solvable subgroup $G$ of $\diffh{}{}$ satisfies 
\[ \dim G = \dim \overline{G} \leq 2 \] 
by the formal classification of such groups
(cf. \cite[Theorem 5.3]{rib:cimpa}  \cite[section $6B_{3}$]{Ilya-Yako}).
\end{rem}
\begin{rem}
Notice that solvable subgroups of $\diffh{}{n}$ are not in general finite dimensional
for $n \geq 2$ (cf. Remark \ref{rem:fdefd}).
\end{rem}

It is very easy to use the extension theorems in section \ref{sec:fdgfd} to find a big class of examples
of finite dimensional subgroups of formal diffeomorphisms.
\begin{teo}
\label{teo:ext}
Let $G$ be a subgroup of $\diffh{}{n}$. Suppose that $G$ has a subnormal series
\begin{equation}
\label{equ:norp}
 \{Id\} =G^{0} \lhd G^{1} \lhd \ldots \lhd G^{m} = H
\end{equation}
such that $G^{j+1}/G^{j}$ is either
\begin{itemize}
\item finite or
\item abelian and finitely generated in the extended sense or
\item a connected Lie group (cf. Definition \ref{def:lie})
\end{itemize}
for any $0 \leq j < m$. Then $G$ is finite dimensional.
\end{teo}
\begin{proof}
The result is a straightforward consequence of Proposition \ref{pro:elems}, Lemma
\ref{lem:fis} and Propositions \ref{pro:fgaifd} and \ref{pro:lieialge}.
\end{proof}
For instance a finite-by-cyclic-by-cyclic-by-finite-by-cyclic group of formal diffeomorphisms
is finite dimensional. Anyway we think that it is interesting to apply the extension
approach laid out in section \ref{sec:fdgfd} to show
that several distinguished classes of groups are always finite dimensional.
Our first targets are polycyclic groups.
\begin{defi}
\label{def:norp}
Let $G$ be a group. We say that $G$ is {\it polycyclic} if it has a subnormal series
as in Equation (\ref{equ:norp})
such that $G^{j+1}/G^{j}$ is cyclic for any $0 \leq j < m$.
\end{defi}
\begin{rem}
A group $G$ is polycyclic if and only if is solvable and every subgroup of $G$ is finitely generated
(cf. \cite[Theorem 5.4.12]{Robinson}).
\end{rem}
\begin{teo}
\label{teo:vp}
Let $G$ be a virtually polycyclic subgroup of $\diffh{}{n}$. Then
$G$ is finite dimensional.
\end{teo}
\begin{proof}
A virtually polycyclic group is a cyclic-by-$\ldots$-by-cyclic-by-finite group.
%The group $G$ has a finite index normal polycyclic subgroup $H$.
%By definition $H$ has a subnormal series of cyclic extensions.
%Propositions \ref{pro:elems} and \ref{pro:fgaifd} imply that $H$
%is finite dimensional.
Therefore $G$ is finite dimensional by Theorem \ref{teo:ext}.
\end{proof}
\begin{defi}
We say that a group $G$ is {\it supersolvable} if  it has a normal series in which all the
factors are cyclic groups.
\end{defi}
\begin{rem}
The definition is very similar to Definition \ref{def:norp}.
Anyway supersolubility is stronger since we require the groups $G^j$
in Equation (\ref{equ:norp}) to be normal in $G$ for $0 \leq j \leq m$.
\end{rem}
\begin{cor}
\label{cor:sup}
Let $G$ be a virtually supersolvable subgroup of $\diffh{}{n}$.
Then $G$ is finite dimensional.
\end{cor}
\begin{proof}
A virtually supersolvable group is virtually polycyclic. The
result is a consequence of Theorem \ref{teo:vp}.
\end{proof}
Let us focus now on nilpotent groups.
\begin{teo}
\label{teo:vnfgifd}
Let $G$ be a virtually nilpotent subgroup of $\diffh{}{n}$. Suppose
that $G$ is finitely generated in the extended sense.
Then $G$ is finite dimensional. In particular subgroups of $\diffh{}{n}$
of polynomial growth are finite dimensional.
\end{teo}
\begin{proof}
%The group $\overline{G}$ is virtually nilpotent by Lemma \ref{lem:fiscc}.
There exists a subgroup $G_{+}$ of $G$ such that
$G_{+}$ is finitely generated and $\overline{G}_{+} = \overline{G}$
by Lemma \ref{lem:fginf}. Since $G_{+}$ is a subgroup of $G$, it
is virtually nilpotent. Thus up to replace $G$ with $G_{+}$ we can suppose
that $G$ is finitely generated.

Let $H$ be a finite index normal nilpotent subgroup of $G$.
Since $H$ is a finite index subgroup of a finitely generated group,
it is finitely generated (cf. \cite[Theorem 1.6.11]{Robinson}).
Any finitely generated nilpotent group is polycyclic (cf. \cite[Theorem 17.2.2]{karga}).
Therefore $G$ is virtually polycyclic and then finite dimensional by Theorem \ref{teo:vp}.
%
%Therefore $H$ is finite dimensional by Theorem \ref{teo:vp}.
%Since $\dim G = \dim H$ by Lemma \ref{lem:fis}, $G$ is finite dimensional.
%

By a theorem of Gromov \cite{Gromov:growth}, the groups of polynomial growth are exactly
the finitely generated virtually nilpotent groups. Hence every subgroup of $\diffh{}{n}$
of polynomial growth is finite dimensional.
\end{proof}
We provide a sort of reciprocal of the previous theorem in the setting of unipotent
subgroups of formal diffeomorphisms.
\begin{lem}
\label{lem:fduin}
Let $G$ be a unipotent subgroup of $\diffh{}{n}$. Suppose
that $G$ has the finite determination property. Then $G$ is nilpotent.
\end{lem}
\begin{proof}
Since $G$ has the finite determination property,  there exists $k \in {\mathbb N}$
such that $\pi_k : G \to G_{k}^{*}$ is an isomorphism of groups.
Moreover $G_{k}^{*}$ is a unipotent algebraic matrix group.
 Unipotent groups of matrices are always triangularizable and then
nilpotent by Kolchin's theorem (cf. \cite[chapter V, p. 35]{Serre.Lie}).
Hence $G$ is nilpotent.
\end{proof}
\begin{cor}
Let $G$ be a unipotent subgroup of $\diffh{}{n}$. Suppose
that $G$ is finitely generated in the extended sense 
and finitely determined. Then $G$ is finite dimensional.
\end{cor}
\begin{proof}
The group $G$ is nilpotent by Lemma \ref{lem:fduin}. Thus it is finite dimensional 
by Theorem \ref{teo:vnfgifd}.
\end{proof}
\begin{rem}
Theorems \ref{teo:vp}, \ref{teo:vnfgifd} and Corollary \ref{cor:sup} admit straightforward
generalizations to extensions. For instance a finitely generated (in the extended sense)
virtually nilpotent extension of groups of formal diffeomorphisms is finite dimensional.
\end{rem}
Let us focus on solvable subgroups of $\diffh{}{n}$ that are not necessarily polycyclic.
In order to show $\dim G < \infty$ for a solvable subgroup of $\diffh{}{n}$ it suffices
to consider finite generation properties on the derived groups of $G$.
\begin{pro}
\label{pro:solcha}
Let $G$ be a solvable subgroup of $\diffh{}{n}$. Suppose that
$G^{(\ell)}/G^{(\ell +1)}$ is finitely generated in the extended sense for any
$\ell \in {\mathbb Z}_{\geq 0}$. Then $G$ is finite dimensional.
\end{pro}
%The notation $G^{(\ell)}$ stands for the $\ell$-th derived group of $G$ where
%$G^{(0)} = G$ and $G^{(1)}$ is the derived group $G'$.
\begin{proof}
Every extension of the form $G^{(\ell)}/G^{(\ell +1)}$ is abelian and then finite dimensional
by Proposition \ref{pro:fgaifd}. Since there exists $\ell$ such that $G^{(\ell)}=\{1\}$, $G$
is finite dimensional by Theorem \ref{teo:ext}.
\end{proof}
In the following our  goal is  substantially  weakening the finite generation hypotheses in
Proposition \ref{pro:solcha}.
\begin{defi}
We denote
\[ G_{s} = \{ \phi \in G : \phi = \phi_{s} \}, \ G^{s} = \{ \phi_{s} : \phi \in G \}, \]
\[ G_{u} = \{ \phi \in G : \phi = \phi_{u} \}, \ G^{u} = \{ \phi_{u} : \phi \in G \}, \]
$\overline{G}_{s} = (\overline{G})_{s}$
and $\overline{G}_{u} = (\overline{G})_{u}$.
\end{defi}
Notice that $G_s$ (resp. $G_u$)
is the set of semisimple (resp. unipotent) elements of $G$ whereas
$G^{s}$ (resp. $G^u$)  is the set of semisimple (resp. unipotent) parts of elements of $G$.
We can have $G^s \subsetneq G$ (resp. $G^u \subsetneq G$)
if the group $G$ is not splittable.

The next results are intended to show that under certain hypotheses a virtually solvable
subgroup $G$ of $\diffh{}{n}$ is finite dimensional if and only if $G_u$ is finite dimensional.
\begin{lem}
\label{lem:iatfis}
Let $H$ be a finite index normal subgroup of a subgroup $G$ of $\diffh{}{n}$.
Then $\overline{G}_{u} = \overline{H}_{u}$ and
$\overline{\langle G^{u} \rangle} = \overline{\langle H^{u} \rangle}$.
Moreover if $G$ is virtually solvable then $G_{u}$ and $\overline{G}_{u}$
are solvable groups.
\end{lem}
\begin{proof}
Since $\overline{H}$ is a finite index normal solvable subgroup of
$\overline{G}$ by Lemma \ref{lem:lficl},
we obtain $\overline{H}_{0} = \overline{G}_{0}$ by Lemma \ref{lem:fiscg}.
The unipotent elements of $\overline{G}$ (resp. $\overline{H}$) are
contained in  $\overline{G}_{0}$ (resp. $\overline{H}_{0}$) by Remark \ref{rem:closu}.
Hence we obtain $\overline{G}_{u} = \overline{H}_{u}$.

Given $\alpha \in G^{u}$, there exists $k \in {\mathbb N}$ such that
$\alpha^{k} \in H^{u}$.
Since
\[ \alpha \in \overline{\langle \alpha^{k} \rangle}  =
\{ \alpha^{t} : t \in {\mathbb C} \}, \]
we obtain
$\alpha \in  \overline{\langle \alpha^{k} \rangle} \subset  \overline{\langle H^{u} \rangle}$.
We deduce
$\overline{\langle G^{u} \rangle} \subset \overline{\langle H^{u} \rangle}$.
It is clear that $\overline{\langle H^{u} \rangle} \subset \overline{\langle G^{u} \rangle}$.
Thus  $\overline{\langle G^{u} \rangle} = \overline{\langle H^{u} \rangle}$ holds.

Let us show that $\overline{G}_{u}$ is a group if $G$ is virtually solvable.
By the first part of the proof we can suppose that $G$ is solvable.
The group $\overline{G}_{0}$ is solvable by Lemma \ref{lem:fiscc} and
$\overline{G}_{u} \subset \overline{G}_{0}$ by Remark \ref{rem:closu}.
Moreover since $G_{1,u}  \subset G_{1,0}$ and the latter group is solvable and connected,
we can suppose, up to a linear change of coordinates,
that all elements of $\overline{G}_{u}$ have linear parts that
are unipotent upper triangular matrices by Lie-Kolchin's theorem. Since the set of
unipotent upper triangular matrices is a group we deduce that $\overline{G}_{u}$
and then $G_{u}$ are groups.
\end{proof}
\begin{teo}
\label{teo:vsfg}
Let $G$ be a virtually solvable subgroup of $\diffh{}{n}$
such that $G$ is finitely generated over $G_{u}$ in the extended sense.
Then $G/G_u$ is finite dimensional.
In particular $G$ is finite dimensional if and only if $G_{u}$ is finite dimensional.
%
%Suppose that $G_{u}$ is finitely generated in the extended sense and nilpotent.
%Then $G$ is finite dimensional.
\end{teo}
\begin{cor}
\label{cor:vsfg}
Let $G$ be a finitely generated virtually solvable subgroup of $\diffh{}{n}$
such that $G_{u}$ is finitely generated and nilpotent.
Then $G$ is finite dimensional.
\end{cor}
%\begin{rem}
%\label{rem:gusg}
%We claim that if
%$G$ is a virtually solvable subgroup of $\diffh{}{n}$ then $G_{u}$ is a group.
%Indeed $\overline{G}_{0}$ is solvable by Lemma \ref{lem:fiscc} and
%$\overline{G}_{u} \subset \overline{G}_{0}$ by Remark \ref{rem:closu}.
%Moreover since $(G_{1})_{u}  \subset G_{1,0}$ and the latter group is solvable and connected,
%we can suppose, up to a linear change of coordinates,
%that all elements of $\overline{G}_{u}$ have linear parts that
%are unipotent upper triangular matrices by Lie-Kolchin's theorem. Since the set of
%unipotent upper triangular matrices is a group we deduce that $\overline{G}_{u}$
%and then $G_{u}$ are groups.
%\end{rem}
\begin{rem}
Notice that the nilpotence of $G_u$ is necessary in Corollary \ref{cor:vsfg}
by Lemma \ref{lem:fduin}.
The main advantage of Corollary \ref{cor:vsfg} is that we are replacing a property
of finite generation for every derived subgroup of $G$ by the analogous property
for just $G$ and $G_u$. Both conditions of finite generation can be replaced by
finite generation in the extended sense.
\end{rem}
\begin{proof}[Proof of Theorem \ref{teo:vsfg}]
The set $G_{u}$ is a subgroup of $G$ by Lemma \ref{lem:iatfis}.
Thus $G_{u}$ is a normal subgroup of $G$.

There exists a subgroup $J$ of $G$ such that
$G_{u} \subset J$, $\overline{J} = \overline{G}$ and
$J$ is finitely generated over $G_{u}$ by Lemma \ref{lem:fginf}.
It suffices to show $\dim (J/G_u) < \infty$
since $\dim (G/G_u) = \dim (J/G_u)$.

We denote $K= J \cap \overline{G}_{0}$.
The group $K$ is a finite index normal subgroup of $J$.
Since $\overline{G}_{0}$ is solvable (by Lemma \ref{lem:fiscc}) and
$G_{1,0}$ is connected,
we can suppose that all elements of $\overline{G}_{0}$ have linear parts
that belong to the group of upper triangular matrices.
Notice that $G_{u}$ is contained in $\overline{G}_{0}$.
Since $K/G_{u}$ is a finite index normal subgroup of $J/G_{u}$,
the group $K/G_{u}$ is finitely generated.
The elements of the derived group $K'$ have linear parts that are unipotent upper
triangular matrices. Thus $K'$ is contained in $G_{u}$ and
$K/G_{u}$ is abelian. We deduce $\dim (K/G_{u}) < \infty$ by
Proposition \ref{pro:fgaifd} and then
$\dim (J/G_u) < \infty$ by Corollary \ref{cor:fis}.
\end{proof}
\begin{proof}[Proof of Corollary \ref{cor:vsfg}]
Since $G_u$ is a normal subgroup of $G$ by Lemma \ref{lem:iatfis}, it suffices to show
$\dim G_u < \infty$ by Theorem \ref{teo:vsfg} and
Proposition \ref{pro:elems}. The group $G_u$ is finite dimensional by
Theorem \ref{teo:vnfgifd}.
\end{proof}
Let us see that the finite generation of $G$ can be dropped in Theorem \ref{teo:vsfg}
if we consider a splittable group.
\begin{teo}
\label{teo:vssp}
Let $G$ be a splittable virtually solvable subgroup of $\diffh{}{n}$.
Then $G/G_{u}$ is finite dimensional and $\overline{G}_{u} = \overline{(G_{u})}$ .
In particular $G$ is finite dimensional if and only if $G_{u}$ is finite dimensional.
\end{teo}
\begin{cor}
\label{cor:vssp}
Let $G$ be a splittable virtually solvable subgroup of $\diffh{}{n}$.
Suppose that either
\begin{itemize}
\item $G_u$ is nilpotent and finitely generated in the extended sense or
\item $\{ \phi^{t} : t \in {\mathbb C} \} \subset G$ for any $\phi \in G_u$ and
$G_u$ has the finite determination property.
\end{itemize}
Then $G$ is finite dimensional.
\end{cor}
We will use the following theorem by Togo.
\begin{teo}[\cite{togo}]
\label{teo:togo}
Let $L$ be a Zariski-connected solvable subgroup of
$\mathrm{GL}(n, k)$. Then $L_{u}$ is a Zariski-closed normal subgroup of
$L$. Moreover if $L$ is splittable then
$\overline{L}_{u} = \overline{(L_{u})}$.
\end{teo}
\begin{proof}[Proof of Theorem \ref{teo:vssp}]
The group $H := G \cap \overline{G}_{0}$ is a finite index normal subgroup
of $G$.
Thus $\overline{H}$ is a finite index normal subgroup of $\overline{G}$
by Lemma \ref{lem:lficl}. In particular $\overline{H}$ contains $\overline{G}_{0}$
by  Lemma \ref{lem:fiscg}. Since $\overline{H} \subset \overline{G}_{0}$,
we obtain $\overline{H} = \overline{G}_{0}$.
Moreover $\overline{G}_{0}$ is solvable by Lemma \ref{lem:fiscc} and then $H$ is solvable.
%The group $G_u$ is contained in $\overline{G}_{0}$ and it is also contained in
%$G$ since $G$ is splittable. In particular $G_u$ is contained in $H$ and thus we obtain
%$G_u=H_u$.
Since $G$ is splittable and $G_u \subset G \cap \overline{G}_{0}$, we obtain
$G_u \subset H$. We deduce that $H$ is splittable and $H_u=G_u$.
Up to replace $G$ with $H$ if necessary we can suppose
$\overline{G} = \overline{G}_{0}$ and $G$ is solvable
by Lemma \ref{lem:iatfis}.

%The set $G_{u}$ is a subgroup of $G$ by Lemma \ref{lem:iatfis}.
%The group $H := G \cap \overline{G}_{0}$ is a finite index normal subgroup
%of $G$. Thus $\overline{H}$ is a finite index normal subgroup of $\overline{G}$
%by Lemma \ref{lem:lficl}. In particular $\overline{H}$ contains $\overline{G}_{0}$
%by  Lemma \ref{lem:fiscg}. Since $\overline{H} \subset \overline{G}_{0}$,
%we obtain $\overline{H} = \overline{G}_{0}$.
%Moreover since $\overline{G}_{u} \subset \overline{G}_{0}$,
%the set of unipotent parts of elements of $G$ is contained in both $G$
%(by hypothesis) and $\overline{G}_{0}$ and then in $H$. Thus $H$ is splittable
%and $G_u = H_u$. Moreover since $\overline{G}_{0}$ is solvable by Lemma
%\ref{lem:fiscc}, $H$ is solvable. It suffices to show that $H/H_u$ is finite dimensional.

Fix $k \in {\mathbb N}$.
Since $G$ is splittable, the group $G_{k}^{*}$ is splittable.
The group $\overline{G}$ is solvable
(Lemma \ref{lem:pap}) and then $G_{k}$ is solvable.
The property $\overline{G}=\overline{G}_{0}$ implies $G_{k}=G_{k,0}$.
Since the Zariski-closure of $G_{k}^{*}$ is solvable and
a (connected) smooth irreducible algebraic set,
$G_{k}^{*}$ is a splittable Zariski-connected solvable group.
%A theorem of Togo \cite{togo} implies that a Zariski-connected solvable subgroup $L$ of
%$\mathrm{GL}(n, k)$ satisfies that $L_{u}$ is a Zariski-closed invariant subgroup of
%$L$ and that moreover if $L$ is splittable then
%$\overline{L}_{u} = \overline{(L_{u})}$.
We can apply Togo's Theorem \ref{teo:togo} to the subgroup $G_{k}^{*}$ of $G_{k}$.
We obtain
\[ G_{k,u} = (\overline{G_{k}^{*}})_{u} = \overline{(G_{k}^{*})_{u}} =
\overline{(G_{u})_{k}^{*}} = G_{u,k} \]
for any $k \in {\mathbb N}$. Since $\overline{G}_{u} = \varprojlim  G_{k,u} $
by Remark  \ref{rem:u} and $\overline{(G_{u})} = \varprojlim  G_{u,k}$ by definition, the
equality $\overline{G}_{u} =\overline{(G_{u})}$ holds.
%This leads us to
%$\overline{H}_{u} = \varprojlim  (H_{k})_{u} = \varprojlim  (H_{u})_{k} =\overline{(H_{u})}$.
We have
\[ \dim G/G_u = \dim \overline{G}/ \overline{(G_{u})} =
\dim \overline{G}/ \overline{G}_{u} . \]
The map
$\hat{\pi}_{1} :  \overline{G}/ \overline{G}_{u}  \to G_{1}/ G_{1,u}$ is injective
and thus an isomorphism of groups
by Remark  \ref{rem:u}. Therefore $\overline{G}/ \overline{G}_{u}$
is finite dimensional by Proposition \ref{pro:elem} and
$\dim \overline{G}/ \overline{G}_{u} = \dim G_{1} - \dim G_{u,1}$.
We obtain
$\dim G/ {G}_{u} = \dim G_{1} - \dim G_{u,1} < \infty$.
\end{proof}
\begin{proof}[Proof of Corollary \ref{cor:vssp}]
It suffices to show that $G_u$ is finite dimensional by Theorem \ref{teo:vssp}.
In the former case it is a consequence of Theorem \ref{teo:vnfgifd}
In the later case we can apply Corollary \ref{cor:charfdifd} since
$\overline{\langle \phi \rangle}= \{  \phi^{t} : t \in {\mathbb C} \}$
for any $\phi \in \diff{u}{n}$ (cf. Remark \ref{rem:closu}).
\end{proof}
\subsection{Examples of infinitely dimensional groups}
So far we exhibited distinguished families of virtually solvable groups whose members
are finite dimensional. Now let us consider the problem of finding
infinite dimensional families
of solvable subgroups of $\diffh{}{n}$.
\begin{rem}
\label{rem:eosnfd}
Consider the subgroup $\langle \phi, \eta \rangle$ of $\diff{}{2}$ generated by 
$\phi(x,y) = (x, y(1+x))$ and $\eta (x,y) = (x, y+ x^{2})$. Since ${\langle \phi, \eta \rangle}'$  is contained in the 
group $H_{1}:=\{ (x, y+ b (x)) : b \in {\mathbb C}\{x\} \cap (x^{2}) \}$, we get
${\langle \phi, \eta \rangle}^{(2)} = \{Id\}$.
In particular $\langle \phi, \eta \rangle$ is a finitely generated unipotent solvable subgroup of $\diff{}{2}$.
Since $[(x, y - x^{k}), \phi] = (x, y + x^{k+1})$
for any $k \in {\mathbb N}$, the group $\langle \phi, \eta \rangle$ is not nilpotent. Thus 
$\langle \phi, \eta \rangle$ is neither finitely determined nor finite dimensional by Lemma \ref{lem:fduin}.
\end{rem}
Next we  see that solvable subgroups of $\diffh{}{n}$ of
high derived length are never finite dimensional.
\begin{pro}
\label{pro:infdim}
Let $G$ be a solvable group contained in $\diffh{u}{n}$ whose derived length is greater
that $n$. Then $G$ does not have the finite determination property.
In particular $G$ is not finite dimensional.
\end{pro}
\begin{rem}
Such groups always exists if $n \geq 2$. Indeed
the maximum of the derived lengths of the solvable unipotent subgroups of $\diffh{}{n}$
is $2n-1$ \cite{JR:arxivdl}.
\end{rem}
\begin{rem}
Notice that given a solvable group $G$, its derived length
is the supremum of the derived lengths of all its finitely generated subgroups.
Hence there exists a finitely generated
subgroup $H$ of $G$ with the same derived length. In particular we can suppose that
the examples provided by Proposition \ref{pro:infdim} for $n \geq 2$ are finitely generated.
\end{rem}
\begin{rem}
Let us provide an example of a group that satisfies the hypotheses of Proposition \ref{pro:infdim}
\cite{JR:arxivdl}. We denote 
$\phi(x,y) = (x, y(1+x))$, $\eta (x,y) = (x, y+ x^{2})$ and $\psi (x,y) = \left( \frac{x}{1-x},y \right)$.
Consider the subgroup $G:= \langle \phi, \eta, \psi \rangle$ of $\diff{}{2}$. 
We define the subgroup 
\[ H_0 = \{ x, y (1+a(x)) + x b (x)) : a, b  \in {\mathbb C}\{x\} \cap (x)  \}  \]
of $\diff{}{2}$. It is clear that $G'$ is contained in $H_0$ and $G^{(2)}$ is contained in the abelian group 
$H_1$ defined in Remark \ref{rem:eosnfd}. Hence $G$ is a finitely generated unipotent 
solvable subgroup of $\diff{}{2}$ whose derived length is at most $3$.  Since
\[ [\psi, \phi] = \left( x, y \frac{1+2x}{(1+x)^{2}} \right), \ \ 
[\eta^{-1}, \phi] = \left( x, y  + x^{3} \right) \ \mathrm{and} \]
\[ [[\psi, \phi] , [\eta^{-1}, \phi] ] = \left( x , y - \frac{x^5}{(1+x)^{2}} \right), \]
the diffeomorphism $ [[\psi, \phi] , [\eta^{-1}, \phi] ]$ belongs to $G^{(2)} \setminus \{ Id\}$ 
and hence the derived length of $G$ is equal to $3$. 
\end{rem}
\begin{proof}[Proof of Proposition \ref{pro:infdim}]
Suppose that $G$ has the finite determination property.
Hence $G$ is nilpotent by Lemma \ref{lem:fduin}. Therefore
the derived length of $G$ is less or equal than $n$ \cite[Theorem 5]{JR:arxivdl},
obtaining a contradiction.
\end{proof}
\section{Local intersection theory}
\label{sec:locint}
Let us explain in this section why Theorem \ref{teo:main} holds. A priori we could use
directly Binyamini's theorem \cite{Binyamini:finite} since an algebraic group is a Lie group
with finitely many connected components. Anyway we think that it is instructive to apply
our canonical approach to the ideas introduced by Seigal-Yakovenko in
\cite{Seigal-Yakovenko:ldi} (to show Theorem \ref{teo:main} for finitely generated
in the extended sense abelian subgroups of formal diffeomorphisms).

Consider two formal subschemes $I$ and $J$ of the scheme
$\mathrm{spec} \ \hat{\mathcal O}_{n}$. We can identify $I$ and $J$ with two
ideals of the ring $\hat{\mathcal O}_{n}$ of formal power series.
%Let us suppose that $I$ and $J$ have complementary pure dimensions.
\begin{defi}
\label{def:intmul}
We define the {\it intersection multiplicity} $(I,J)$ as
\[ (I,J) = \dim_{\mathbb C} \hat{\mathcal O}_{n}/(I+J) . \]
\end{defi}
\begin{rem}
This definition of intersection multiplicity coincides with the usual one  if
$I$ and $J$ are complete intersections of complementary dimension. 
It is finite if and only if
the usual intersection multiplicity is finite. Moreover it provides un upper bound
for the usual intersection multiplicity (cf. \cite[Proposition 8.2]{Fulton:inter}).
Therefore by showing Theorem \ref{teo:main} with Definition \ref{def:intmul}
it will be automatically satisfied for the usual intersection multiplicity.
\end{rem}
Next let us show Theorem \ref{teo:main}. Since we follow Seigal-Yakovenko's ideas we
refer to their paper \cite{Seigal-Yakovenko:ldi} for details. We are interested in stressing
how their point of view fits in the context of the theory of finite dimensional groups of formal
diffeomorphisms.
\begin{proof}[Proof of Theorem \ref{teo:main}]
Let $V$ and $W$ be formal subschemes of $\mathrm{spec} \ \hat{\mathcal O}_{n}$.
%of complementary pure dimensions.
Suppose that $V$ is given by the ideal $K$ of $\hat{\mathcal O}_{n}$.
Given $\phi \in \diff{}{n}$ the subscheme $\phi^{-1}(V)$ is given by the ideal
$\phi^{*} K = \{ f \circ \phi : f \in K \}$.

There exists $k \in {\mathbb N}$ such that $\hat{\pi}_{k}: \overline{G} \to G_{k}$
is an isomorphism of groups by Proposition \ref{pro:elem}.
The map $\pi_{m,k}: G_m \to G_k$ is an isomorphism of algebraic groups for any
$m \geq k$. In particular the affine coordinate rings ${\mathbb C}[G_{k}]$ and
${\mathbb C}[G_{m}]$ are isomorphic as ${\mathbb C}$-algebras for any
$m \geq k$.

Given an ideal $J$ of $\hat{\mathcal O}_{n}$ the property
$\dim_{\mathbb C} \hat{\mathcal O}_{n} / J > m$ is equivalent to
a system of algebraic equations on the coefficients of the $m$-th jets of the
generators of $J$ \cite[Lemma 3]{Seigal-Yakovenko:ldi}.
In particular $(\phi^{-1}(V), W) > m$ holds for $\phi \in \diffh{}{n}$
if and only if the coefficients of the $m$-th jet of
$\phi$ satisfy a certain system of algebraic equations.
More intrinsically we can say that $S_{m} := \{\phi \in \overline{G} : (\phi^{-1} (V), W) > m\}$
defines an ideal $I_{m}$ of the affine coordinate ring ${\mathbb C}[G_{m}]$.
It also defines an ideal, that we denote also by $I_{m}$, in ${\mathbb C}[G_{k}]$
for any $m \geq k$. We can suppose $I_{m} \subset I_{m'}$ for all $m' \geq m \geq k$
by replacing $I_{m}$ with $I_{k}+ \ldots + I_{m}$ for $m \geq k$.
Since ${\mathbb C}[G_{k}]$ is noetherian, there exists $m_{0} \geq k$ such that
$I_{m} = I_{m_{0}}$ for any $m \geq m_{0}$.
In particular  $(\phi^{-1} (V), W) > m_{0}$ implies $(\phi^{-1} (V), W) = \infty$
for any $\phi \in \overline{G}$.
\end{proof}
\begin{rem}
The key point of the proof is showing that the increasing sequence of ideals
$I_{1} \subset I_{2} \subset \ldots$ is contained in a noetherian ring.
Seigal and Yakovenko show that it is contained in a ring of quasipolynomials  in their setting
\cite{Seigal-Yakovenko:ldi} whereas Binyamini includes them in a noetherian subring of
continuous functions of $G$ \cite{Binyamini:finite}.
We use that the coefficients of degree greater than $k$
of the Taylor expansion of the elements of $G$ are regular functions on the coefficients of
degree less or equal than $k$ if $G$ is finite dimensional (Remark \ref{rem:hotlot}).
This allows to write all equations defining the ideals $I_{m}$
in terms of the coefficients of $\phi \in G$ of degree less or equal than $k$.
In particular Theorem \ref{teo:main}
is an immediate  consequence of the noetherianity of polynomial rings in finitely many complex variables.
More precisely, in the finite dimensional setting the noetherian ring ${\mathbb C}[G_{k}]$ containing al the ideals
$I_{m}$ for $m \in {\mathbb N}$ is an affine coordinate ring of an algebraic matrix group
canonically associated to $G$.
\end{rem}
\begin{proof}[Proof of Theorem \ref{teo:mainc}]
The hypothesis implies that $G$ is finite dimensional by Theorems \ref{teo:vp} and \ref{teo:vnfgifd}.
Hence the conclusion is a consequence of Theorem \ref{teo:main}.
\end{proof}
%
%
%Binyamini shows that the property $(\phi(I), J) > m$ defines an ideal in some noetherian
%${\mathbb C}$-algebra $R_{\rho}$ of continuous functions of $G$. Since $(\phi(I), J) > m+1$
%implies $(\phi(I), J) > m$, he obtains an increasing sequence (indexed by $m \in {\mathbb N}$) of ideals
%and the conclusion of Theorem \ref{teo:main} is a consequence of the noetherianity of $R_{\rho}$.
%In our context we can apply the same strategy to show Theorem \ref{teo:main} by using
%a simpler noetherian ring.
%In section \ref{sec:locint} we will see that given a finite dimensional subgroup $G$
%of $\diffh{}{n}$ there exists $k \in {\mathbb N}$ such that
%the property $(\phi(I), J) > m$ defines a system of polynomial equations on the
%coefficients of the $k$-th jet of $\phi \in G$. Hence Theorem \ref{teo:main}
%is an immediate  consequence of the noetherianity of polynomial rings in finitely many complex variables.
%
%
\appendix
\section{Solvable groups of formal diffeomorphisms}
In this section we expand  the study of virtually solvable subgroups of formal
diffeomorphisms of section \ref{sec:ffdg}. Our approach is based on considering
extensions of groups and hence it makes sense to generalize the results in
section \ref{sec:ffdg} to extensions. But we also would like to understand better
the phenomenon described in Theorem \ref{teo:vssp}: A splittable virtually solvable
subgroup $G$ of $\diffh{}{n}$ satisfies $\dim G < \infty$ if and only if $\dim G_u < \infty$.
A natural question is whether or not this ``reduction to unipotent" is possible
somehow in general or it is peculiar of the splittable case.
We will find analogues for the general case (Corollaries \ref{cor:zuuz} and \ref{cor:gfdgufd})
and even for extensions (Theorem \ref{teo:gfdgufd}).
%
%We identify new type of extensions that are finite dimensional
%in the setting of virtually solvable subgroups of $\diffh{}{n}$.
%
%
%
%Given $\phi \in \diffh{}{} \setminus \{Id\}$ we define its fixed point index by
%$\min \{ k \in {\mathbb N}: \phi_{k} \not \equiv Id \}$.
%The definition makes sense since its coincide with the usual definition of
%fixed point index of $\phi$ at $0$ if $\phi \in \diff{}{}$.
%
%
%It is well-known that $G$ is non-solvable if and only if it contains two
%elements $\phi, \eta$ with different order of contact
%with the identity such $D_{0} \phi \equiv D_{0} \eta \equiv Id$.
%In such a case $G$ contains elements with arbitrarily high
%fixed point index.
%
%
%Clearly a solvable subgroup $G$ of $\diffh{}{}$ satisfies the finite determination property.
%
%Let us show $(2) \implies (1)$.
%The map $\pi_{k} : G \to D_{k}$ is injective for some $k \in {\mathbb N}$
%by hypothesis. Every element $\phi$ of the derived group $G'$ satisfies
%$D_{0} \phi \equiv Id$.
%The group $\pi_{k}(G')$ is a matrix group of unipotent elements and hence nilpotent
%by Kolchin's theorem. We deduce that $G'$ is nilpotent and then $G$ is solvable.
\subsection{Finite dimensional subextensions}
Solvable groups in dimension $1$ are finite dimensional by Proposition \ref{pro:fd1fd}.
It is natural to study solvable groups and solvable extensions in higher dimensions
with the purpose of understanding how far they are of being finite dimensional or
also where it is concentrated the possible infinite dimensional nature of such groups.
These topics are the subjects of this section.

Let us explain our goal   a little bit more precisely.
Let $H$ be a normal subgroup of a solvable subgroup $G$ of $\diffh{}{n}$.
Suppose that $G/H$ is virtually solvable.
We already know that in general the extension $G/H$ is not finite dimensional
but anyway we want to find ``unipotent" extensions
%subextensions of $G/H$ that are finite dimensional or
that are finite dimensional if and only if $G/H$ is.
Indeed we will see that $\dim G/H < \infty$ is equivalent to
$\dim \langle G^{u}, H \rangle/H < \infty$ (Theorem \ref{teo:gfdgufd}).
%In this way we can reduce the problem of
%determining whether or not
%$\dim G/H < \infty$ to the analogous problem
%for the simpler ``unipotent" extension $\langle G^u, H \rangle/H$.
We will also show analogues of Theorems \ref{teo:vsfg} and \ref{teo:vssp} for
extensions.
%Since in general $G^{u}$ is not contained in $G$ and it is easier to work
%with subgroups of $G$, we will provide conditions guaranteeing
%$\dim G/\langle G_u, H \rangle < \infty$ (Theorems \ref{teo:subus} and \ref{teo:rgwgu}).
%In particular in such cases
%$G/H$ is finite dimensional if and only if the subextension $\langle G_u, H \rangle/H$
%is finite dimensional.

Analogously as in Corollary \ref{cor:vssp} we want to use Togo's theorem for virtually
solvable extensions even in the
non-splittable case.
%
% A theorem of Togo \cite{togo} implies that a Zariski-connected solvable subgroup $L$ of
%$\mathrm{GL}(n, k)$ satisfies that $L_{u}$ is a Zariski-closed invariant subgroup of
%$L$ and that moreover if $L$ is splittable then
%$\overline{L}_{u} = \overline{(L_{u})}$ (and hence
%$\overline{L}_{u} = \overline{\langle L^{u} \rangle}$).
%Since $\diffh{}{n}$ is a projective limit of algebraic matrix groups, we will take
%profit of Togo's theorem in the setting of groups of formal diffeomorphisms.
%Anyway, subgroups of $\diffh{}{n}$ are not in general splittable.
The next linear results
are intended to address this issue.
\begin{lem}
\label{lem:gungs}
Let $G$ be a subgroup of $\mathrm{GL}(n,{\mathbb C})$.
Then $\langle G^{u} \rangle$ is normal in $\langle G^{s}, G^{u} \rangle$.
\end{lem}
\begin{proof}
Every element of $G^{u}$ is of the form $A_{u}$ for some $A \in G$.
Given $B \in G$ the property $B A B^{-1} \in G$ implies
$B A_{u} B^{-1} = (B A B^{-1})_{u} \in G^{u}$.
It is clear that $B_{u} A_{u} B_{u}^{-1} \in \langle G^{u} \rangle$ and that
\[ B_{s} A_{u} B_{s}^{-1} = B_{u}^{-1} (B A_{u} B^{-1}) B_{u} \in \langle G^{u} \rangle \]
for any $B \in G$. We deduce $C A_{u} C^{-1} \in \langle G^{u} \rangle$ for all
$A_{u} \in G^{u}$ and $C \in \langle G^{s}, G^{u} \rangle$.
Hence we get $C \langle G^{u} \rangle C^{-1} \subset \langle G^{u} \rangle$
for any $C \in \langle G^{s}, G^{u} \rangle$.
%This implies $C \langle G^{u} \rangle C^{-1} = \langle G^{u} \rangle$
%for any $C \in \langle G^{s}, G^{u} \rangle$.
\end{proof}
\begin{lem}
\label{lem:uwgsgu}
Let $G$ be a subgroup of $\mathrm{GL}(n,{\mathbb C})$.
Then any element of $\langle G^{s} , \overline{\langle G^{u} \rangle} \rangle$ is
of the form $\alpha \beta$ for some $\alpha \in G^{s}$ and
$\beta \in \overline{\langle G^{u} \rangle}$.
\end{lem}
\begin{proof}
We denote $J = \langle G^{s} , \overline{\langle G^{u} \rangle} \rangle$.
Since $G^{s}$ normalizes $\langle G^{u} \rangle$ by Lemma \ref{lem:gungs},
$G^{s}$ normalizes $\overline{\langle G^{u} \rangle}$. Thus
$\overline{\langle G^{u} \rangle}$ is a normal subgroup of $J$. We deduce that
any element of $J$ is of the form  $\alpha_{1} \ldots \alpha_{m} \beta'$ where
$\alpha_{1}, \ldots, \alpha_{m} \in G^{s}$
and $\beta' \in \overline{\langle G^{u} \rangle}$. The element $\alpha_{j}$ is equal to
$(\gamma_{j})_{s}$ for some $\gamma_{j} \in G$ and any $1 \leq j \leq m$.
Since $\langle G^{u} \rangle$ is normal in $\langle G^{s}, G^{u} \rangle$
by Lemma \ref{lem:gungs} we obtain
\[ \alpha_{1} \ldots \alpha_{m} = \gamma_{1} \ldots \gamma_{m} \beta'' \]
for some $\beta'' \in \langle G^{u} \rangle$.
We define $\alpha = (\gamma_{1} \ldots \gamma_{m})_{s}$ and
$\beta = (\gamma_{1} \ldots \gamma_{m})_{u} \beta'' \beta'$.
It is clear that $\alpha$ belongs to $G^{s}$ and $\beta$ belongs to
$\overline{\langle G^{u} \rangle}$.
\end{proof}
The Zariski-closure of a virtually solvable matrix group $G$ is splittable by Chevalley's
Theorem \ref{teo:chevj}. Anyway in the following lemma we provide another
extension of $G$ that is splittable and contained in $\overline{G}$.
The main advantage is that in the new extension we can characterize
its unipotent elements in terms of $G^{u}$.
\begin{lem}
\label{lem:jvs}
Let $G$ be a virtually solvable subgroup of $\mathrm{GL}(n,{\mathbb C})$.
Then
$\langle G^{s} , \overline{\langle G^{u} \rangle} \rangle^{u} = \overline{\langle G^{u} \rangle}$.
In particular
$\langle G^{s} , \overline{\langle G^{u} \rangle} \rangle$ is splittable.
%$\gamma_{u} \in  \overline{\langle G^{u} \rangle}$ for any
%$\gamma \in \langle G^{s} , \overline{\langle G^{u} \rangle} \rangle$. Moreover we have .
\end{lem}
\begin{proof}
%There exists a finite index normal solvable subgroup $H$ of $G$.
%Therefore $\overline{H}$ is a finite index normal solvable subgroup of
%$\overline{G}$ by Lemma \ref{lem:lficl}.
%In particular $\overline{G}$ is virtually solvable.
The group $\overline{G}$ is virtually solvable by Lemma \ref{lem:fiscc}.
Since $G$ normalizes $\langle G^{u} \rangle$, so it normalizes
$\overline{\langle G^{u} \rangle}$. The normalizer of an algebraic
group is algebraic and thus $\overline{\langle G^{u} \rangle}$ is a normal
subgroup of $\overline{G}$.

We denote $J = \langle G^{s} , \overline{\langle G^{u} \rangle} \rangle$.
Fix $\gamma \in J$. Let us show $\gamma_{u} \in  \overline{\langle G^{u} \rangle}$.
 We can suppose that $\gamma$ is of the form
$\alpha \beta$ for some $\alpha \in G^{s}$ and
$\beta \in  \overline{\langle G^{u} \rangle}$ by Lemma \ref{lem:uwgsgu}.
There exists $k \in {\mathbb N}$ such that
$\alpha^{k} \in \overline{\langle \alpha \rangle}_{0}$ where
$\overline{\langle \alpha \rangle}_{0}$ is the connected component of
$Id$ of $\overline{\langle \alpha \rangle}$.
Since $\overline{\langle \alpha^{k} \rangle}$ is a finite index normal subgroup of 
$\overline{\langle \alpha \rangle}$ by Lemma \ref{lem:lficl}, it contains $\overline{\langle \alpha \rangle}_{0}$
by Lemma \ref{lem:fiscg}.
We obtain $\overline{\langle \alpha^{k} \rangle} =\overline{\langle \alpha \rangle}_{0}$. 
It suffices to show $(\gamma^{k})_{u} \in  \overline{\langle G^{u} \rangle}$.
Indeed since  $\overline{ \langle \eta \rangle} =  \{  \mathrm{exp} (t \log \eta) : t \in {\mathbb C} \}$
for any unipotent element $\eta$ of $\mathrm{GL}(n,{\mathbb C})$, we obtain
\[ \overline{\langle \gamma_{u}^{k}  \rangle} = \overline{\langle \gamma_{u}  \rangle}
= \{ \mathrm{exp} (t \log \gamma_{u}) : t \in {\mathbb C} \}. \]
Thus the property
$\overline{\langle \gamma_{u}^{k}  \rangle} \subset \overline{\langle G^{u} \rangle}$
implies $\gamma_{u} \in \overline{\langle G^{u} \rangle}$.
Hence up to replace $\gamma$ with $\gamma^{k}$ (and as a consequence $\alpha$
with $\alpha^{k}$) we can suppose
$\overline{\langle \alpha \rangle} =\overline{\langle \alpha \rangle}_{0}$.

%Since $\overline{ \langle \eta \rangle} =  \{ \mathrm{exp}  (t \log \eta) : t \in {\mathbb C} \}$
%for any unipotent element $\eta$ of $\mathrm{GL}(n,{\mathbb C})$,
The group $\overline{\langle G^{u} \rangle}$ coincides with the (connected
algebraic) group generated by
$\cup_{\eta \in G^{u}} \{ \mathrm{exp}  (t \log \eta) : t \in {\mathbb C} \}$ by Chevalley's
Theorem \ref{teo:chevc}. Another application of Theorem \ref{teo:chevc} implies that
$L:=\langle \overline{\langle \alpha \rangle}, \overline{\langle G^{u} \rangle} \rangle$
is a connected algebraic group.
The group $L$
is contained in $\overline{G}$ and hence it is virtually solvable.
Since it is connected,
$L$ is solvable. The Lie-Kolchin theorem
%\marginpar{reference}
implies that up to a change
of basis, the group $L$ is upper triangular.
Thus $G^{u}$ and then $\overline{\langle G^{u} \rangle} $
are contained in the group of unipotent upper triangular matrices. In
particular $\overline{\langle G^{u} \rangle}$ is contained in $J_{u}$.
The group $\overline{\langle G^{u} \rangle} $ is normal in $\overline{G}$
and then in $L$ since $L \subset \overline{G}$. We deduce that
every element $\eta$ of $L$ is of the form
$\eta_{1} \eta_{2}$ where $\eta_{1} \in  \overline{\langle \alpha \rangle}$ and
$\eta_{2} \in  \overline{\langle G^{u} \rangle}$.
Since algebraic matrix groups are splittable by Chevalley's Theorem \ref{teo:chevj},
we obtain that $\gamma_{u}$ is of the previous form $\eta_{1} \eta_{2}$
where $\eta_{1} \in  \overline{\langle \alpha \rangle}$ and
$\eta_{2} \in  \overline{\langle G^{u} \rangle}$.
Since $L$ is upper triangular and $\gamma_{u}$ and $\eta_{2}$ are unipotent,
$\eta_{1}$ is unipotent. It is also semisimple since
$\eta_{1} \in \overline{\langle \alpha \rangle}$ and $\alpha$ is semisimple.
We deduce $\eta_{1} \equiv Id$ and then $\gamma_{u} \in \overline{\langle G^{u} \rangle}$.
We obtain $J^{u} \subset  \overline{\langle G^{u} \rangle}$ and since
$ \overline{\langle G^{u} \rangle} \subset J^{u}$ we get
 $J^{u} = J_{u} =  \overline{\langle G^{u} \rangle}$.
\end{proof}
Now let us identify the unipotent elements of the Zariski-closure of a
virtually solvable linear group.
%in Proposition \ref{pro:togo}.
%The next lemma is intended to reduce
%the problem to the case of solvable groups.
%\begin{lem}
%\label{lem:iatfis}
%Let $G$ be a subgroup of $\mathrm{GL}(n,{\mathbb C})$.
%Consider a finite index normal subgroup $H$ of $G$.
%Then $\overline{G}_{u} = \overline{H}_{u}$ and
%$\overline{\langle G^{u} \rangle} = \overline{\langle H^{u} \rangle}$.
%\end{lem}
%\begin{proof}
%Since $\overline{H}$ is a finite index normal solvable subgroup of
%$\overline{G}$ by Lemma \ref{lem:lficl},
%we obtain $\overline{H}_{0} = \overline{G}_{0}$.
%The unipotent elements of $\overline{G}$ (resp. $\overline{H}$) are
%contained in  $\overline{G}_{0}$ (resp. $\overline{H}_{0}$).
%Hence we obtain $\overline{G}_{u} = \overline{H}_{u}$.
%
%Given $\alpha \in G^{u}$, there exists $k \in {\mathbb N}$ such that
%$\alpha^{k} \in H^{u}$.
%Since $\alpha \in \overline{\langle \alpha^{k} \rangle} \subset \overline{\langle H^{u} \rangle}$
%we obtain $\overline{\langle G^{u} \rangle} \subset \overline{\langle H^{u} \rangle}$.
%It is clear that $\overline{\langle H^{u} \rangle} \subset \overline{\langle G^{u} \rangle}$.
%We obtain $\overline{\langle G^{u} \rangle} = \overline{\langle H^{u} \rangle}$.
%\end{proof}
The next proposition is a consequence of the aforementioned Togo's Theorem \ref{teo:togo}
in the setting of virtually solvable (non-necessarily splittable) subgroups of
$\mathrm{GL}(n,{\mathbb C})$.
\begin{pro}
\label{pro:togo}
Let $G$ be a virtually solvable subgroup of $\mathrm{GL}(n,{\mathbb C})$.
Then $\overline{G}_{u} = \overline{\langle G^{u} \rangle}$.
Moreover $\overline{G}_{u}$ is solvable.
\end{pro}
\begin{proof}
There exists a finite index normal subgroup $H$ of $G$ that is solvable and such
that $\overline{H}$ is connected. Up to replace $G$ with $H$ we can suppose
that these properties are satisfied by $G$ by Lemma \ref{lem:iatfis}.

We denote $J = \langle G^{s} , \overline{\langle G^{u} \rangle} \rangle$.
We have $G \subset J \subset  \overline{G}$ and then
$\overline{J}=\overline{G}$.
Since $\overline{G}$ is solvable, $J$ is solvable.
The group $\overline{J}$ is connected and hence  $J$ is Zariski-connected.
%A theorem of Togo implies that a Zariski-connected solvable subgroup $L$ of
%$\mathrm{GL}(n, k)$ satisfies that $L_{u}$ is a Zariski-closed invariant subgroup of
%$L$ and that moreover if $L$ is splittable then
%$\overline{L}_{u} = \overline{(L_{u})}$ (and hence
%$\overline{L}_{u} = \overline{\langle L^{u} \rangle}$).
Since $J$ is splittable by Lemma \ref{lem:jvs}, we obtain
\[ \overline{G}_{u} = \overline{J}_{u} = \overline{( J_{u} )} =
\overline{\overline{\langle G^{u} \rangle}} = \overline{\langle G^{u} \rangle}  \]
by applying Togo's Theorem \ref{teo:togo} to $J$.
%Since $\overline{G}_{u}$ is contained in $\overline{G}$, it is solvable.
\end{proof}
Let us generalize Proposition \ref{pro:togo} to the setting of extensions of groups of
formal diffeomorphisms.
\begin{pro}
\label{pro:zuuz}
Let $H$ be a normal subgroup of a subgroup $G$ of $\diffh{}{n}$.
Suppose that $G/H$ is virtually solvable.
Then $\overline{\langle G^{u}, H \rangle}$ is the closure of
$\langle  \overline{G}_{u}, \overline{H} \rangle$ in the Krull topology.
%Moreover $\overline{G}_{u}$ is a solvable group.
\end{pro}
\begin{proof}
%The group $\overline{G}/\overline{H}$ is virtually solvable by Lemma \ref{lem:claia}.
Since  $G_{k}^{*}/H_{k}^{*}$ is virtually solvable, $\langle G_{k}^{*} , H_k \rangle / H_k$
is virtually solvable. The group $G_{k}/H_{k}$ is the Zariski-closure of
$\langle G_{k}^{*} , H_k \rangle / H_k$ and hence
$G_{k}/H_{k}$ is a virtually solvable algebraic matrix group
for any $k \in {\mathbb N}$ by Lemma \ref{lem:fiscc}.

Consider $\langle G_{k}^{*}, H_k \rangle /H_{k}$ as a subgroup of the algebraic matrix group
$G_{k}/H_k$. Then $\langle (\langle G_{k}^{*}, H_k \rangle /H_{k})^{u} \rangle$ is equal to
$\langle (G_{k}^{*})^{u}, H_k \rangle/H_{k}$.
Let us apply Proposition \ref{pro:togo} to  $\langle G_{k}^{*}, H_k \rangle /H_{k}$.
We deduce
\[ \frac{\overline{\langle (G_{k}^{*})^{u} , H_{k}\rangle }}{H_k} =
{\left( \overline{\left( \frac{\langle G_{k}^{*}, H_k \rangle}{H_{k}} \right)} \right)}_{u} =
{\left( \frac{G_{k}}{H_k} \right)}_{u} = \frac{\langle G_{k,u}, H_k \rangle }{H_k} \]
for any $k \in {\mathbb N}$. We obtain
$\overline{\langle (G_{k}^{*})^{u} , H_{k}\rangle } = \langle G_{k,u}, H_k \rangle$
for any $k \in {\mathbb N}$. It follows that
\[ {\langle H, G^{u} \rangle}_{k} =
\overline{\langle  H_{k}^{*} , (G_{k}^{*})^{u} \rangle}
= \overline{\langle H_k , (G_{k}^{*})^{u} \rangle} = \langle G_{k,u}, H_k \rangle  \]
for any $k \in {\mathbb N}$.
Since
${\langle \overline{G}_{u}, \overline{H} \rangle}_{k}^{*} = \langle G_{k,u}, H_k \rangle
={\langle H, G^{u} \rangle}_{k}$ for any $k \in {\mathbb N}$, the group
$\overline{\langle G^{u}, H \rangle}$ is the closure of
$\langle \overline{G}_{u}, \overline{H} \rangle$ in the Krull topology.
\end{proof}
We can identify the unipotent elements of the Zariski-closure of a virtually solvable
group of formal diffeomorphisms.
\begin{cor}
\label{cor:zuuz}
Let $G$ be a virtually solvable subgroup of $\diffh{}{n}$.
Then $\overline{G}_{u} = \overline{\langle G^{u} \rangle}$.
Moreover $\overline{G}_{u}$ is a solvable group.
\end{cor}
%\begin{proof}
%The group $\overline{G}_{0}$ is solvable by Lemma \ref{lem:fiscc}
%and hence $\overline{G}$ is virtually solvable.
%Hence
%$G_{k} = \overline{\{ \phi_{k} : \phi \in G \}}$ is virtually solvable
%since it is the image of $\overline{G}$ by the morphism of groups
%$\pi_k : \overline{G} \to G_{k}$ for
%$k \in {\mathbb N}$. Thus $\{ \phi_{k} : \phi \in G \}$ is virtually solvable
%for $k \in {\mathbb N}$. By applying Proposition \ref{pro:togo} to
%$\{ \phi_{k} : \phi \in G \}$ we deduce
%$G_{k,u} = \overline{\{ \phi_{k} : \phi \in \langle G^{u} \rangle \}}$
%for any $k \in {\mathbb N}$.
%We obtain $\overline{G}_{u} = \overline{\langle G^{u} \rangle}$
%by the definitions of $\overline{G}_{u}$ and $\overline{\langle G^{u} \rangle}$.
%
% Since $\overline{G}_{u} = \varprojlim G_{k,u}$
%and $G_{j,u} \subset G_{j,0}$ for any $j \in {\mathbb N}$,
%the group
%$\overline{G}_{u}$ is contained in $\overline{G}_{0}$.
%The group $\overline{G}_{0}$ is solvable and thus
%$\overline{G}_{u}$ is solvable.
%\end{proof}
\begin{proof}
We apply Proposition \ref{pro:zuuz} with $H= \{Id\}$.
We obtain that
$\overline{\langle G^{u} \rangle}$ is equal to the closure of
$\langle  \overline{G}_{u} \rangle$ in the Krull topology.
Notice that $\overline{G}_{u}$ is a solvable group by Lemma \ref{lem:iatfis}.
%
%The set $ \overline{G}_{u}$ is contained  in $\overline{G}_{0}$ by
%Remark \ref{rem:closu}.
%Since $(\overline{G}_{0})_{1}^{*}$ is connected
%we can suppose that all elements of $\overline{G}_{0}$ have linear parts that
%are upper triangular up to a linear change of coordinates by Lie-Kolchin's theorem.
%In particular the elements of $\overline{G}_{u}$ are the elements of
%$\overline{G}_{0}$ whose linear parts have all the elements of the principal
%diagonal equal to $1$. We deduce that  $\overline{G}_{u}$ is a group.
%
Since $\overline{G}_{u} =\varprojlim G_{k,u}$, it is closed in the Krull topology.
As a consequence we obtain
 $\overline{G}_{u} = \overline{\langle G^{u} \rangle}$.
%The group $\overline{G}_{0}$ is solvable by Lemma \ref{lem:fiscc}
%and thus $\overline{G}_{u}$ is solvable.
\end{proof}
As a corollary of the previous result we show that in order to
determine the virtually solvable subgroups of $\diffh{}{n}$ that are finite
dimensional, it suffices to consider only groups of unipotent elements.
\begin{cor}
\label{cor:gfdgufd}
Let $G$ be a virtually solvable subgroup of $\diffh{}{n}$.
Then $G$ is finite dimensional if and only of $\langle G^{u} \rangle$
is finite dimensional.
\end{cor}
\begin{proof}
The group $G$ is finite dimensional if and only if $\overline{G}$
has the finite determination property.
Since $\overline{G}_{u}$ contains all the
elements of $\overline{G}$ with identity linear part,
the group $\overline{G}$ has the finite
determination property if and only if $\overline{G}_{u}$ has the
finite determination property.
%The group $\overline{G}_{u}$ is pro-algebraic
%by Proposition \ref{pro:zuuz}.
Since $\overline{G}_{u} = \overline{\langle G^{u} \rangle}$
by Corollary \ref{cor:zuuz}, we deduce that $\overline{G}_{u}$ has
the finite determination property if and only if
$\langle G^{u} \rangle$ is finite dimensional.
\end{proof}
Next we generalize Corollary \ref{cor:gfdgufd} to extensions.
\begin{teo}
\label{teo:gfdgufd}
Let $H$ be a normal subgroup of a subgroup $G$ of $\diffh{}{n}$.
Suppose that $G/H$ is virtually solvable.
Then $\dim G/H < \infty$ if and only if $\dim \langle H, G^{u} \rangle/H < \infty$.
\end{teo}
\begin{proof}
The sufficient condition is clear since $\overline{\langle H, G^{u} \rangle}$
is contained in $\overline{G}$.
Suppose that $\dim \langle H, G^{u} \rangle/H < \infty$.
There exists $k \in {\mathbb N}$ such that the natural map
\[ \hat{\pi}_{k}': \overline{\langle H, G^{u} \rangle} / \overline{H} \to
{\langle H, G^{u} \rangle}_{k} / H_{k} \]
is injective by Proposition \ref{pro:elem}.
Let us show that $\hat{\pi}_{k}: \overline{G} / \overline{H} \to G_{k} / H_{k}$
is injective. This implies that $G/H$ is finite dimensional by Proposition \ref{pro:elem}.

Let $\phi \in \overline{G}$ such that $\hat{\pi}_{k}(\phi \overline{H}) = 1$
or equivalently $\phi_{k} \in H_{k}$.
Since $H_k$ is algebraic, it is splittable by Chevalley's theorem \ref{teo:chevj}.
In particular we get
$\phi_{k,s} \in H_{k}$ and $\phi_{k,u} \in H_k$.
Let $\alpha$ be an element of $\overline{H}$ such that $\alpha_k = \phi_{k,s}$.
The formal diffeomorphism $\alpha^{-1} \circ \phi$
satisfies $(\alpha^{-1} \circ \phi)_k = \phi_{k,u}$ and in particular is unipotent.
Since $\alpha^{-1} \circ \phi$ belongs to $\overline{G}_{u}$ and this group is contained in
$\overline{\langle G^{u}, H \rangle}$ by Proposition \ref{pro:zuuz},
we deduce  $\alpha^{-1} \circ \phi \in \overline{\langle G^{u}, H \rangle}$.
Moreover $\hat{\pi}_{k}' (\alpha^{-1} \circ \phi)=1$ implies
$\alpha^{-1} \circ \phi \in \overline{H}$ by the injective nature of $\hat{\pi}_{k}'$.
We obtain $\phi \in \overline{H}$. Thus $\hat{\pi}_{k}$ is injective.
\end{proof}
Since $\langle H, G_{u} \rangle \subset G$, we want to replace $\langle H, G^{u} \rangle$ with
$\langle H, G_{u} \rangle$ in Theorem \ref{teo:gfdgufd}
to obtain a result analogous to Theorem \ref{teo:vssp}
for extensions.
We impose two conditions in order to accomplish such a task,
namely a ``splitting" property (Theorem \ref{teo:subus}) and
a finite generation one (Theorem \ref{teo:rgwgu}).
\begin{teo}
\label{teo:subus}
Let $H$ be a normal subgroup of a subgroup $G$ of $\diffh{}{n}$.
Suppose that $G/H$ is virtually solvable and
$G^{u} \subset \langle G_u, \overline{H} \rangle$.
%$\langle G, \overline{H} \rangle /\overline{H}$ is splittable.
Then $G/\langle G_{u}, H \rangle$ is finite dimensional.
In particular $G/H$ is finite
dimensional if and only if $\langle G_{u}, H \rangle/H$ is finite dimensional.
\end{teo}
\begin{proof}
Since
%$\langle G, \overline{H} \rangle /\overline{H}$ is splittable we have
$G^{u} \subset \langle G_u, \overline{H} \rangle$, we obtain
\begin{equation}
\label{equ:forudua}
 \overline{\langle G_{u}, H \rangle} = \overline{\langle G_{u}, \overline{H} \rangle}
= \overline{\langle G^{u}, \overline{H} \rangle} =
\overline{\langle G^{u}, H \rangle}.
\end{equation}
The group $\overline{\langle G^{u}, H \rangle}$ is  the closure of
$\langle  \overline{G}_{u}, \overline{H} \rangle$ in the Krull topology
by Proposition \ref{pro:zuuz}. This implies
${\langle G_{u}, H \rangle}_{1}= \langle G_{1,u}, H_{1} \rangle$.

Let us show that
$\hat{\pi}_{1} : \overline{G}/\overline{\langle G_{u}, H \rangle} \to
G_{1}/ {\langle G_u, H \rangle}_{1}$ is injective.
Let $\phi \in \overline{G}$ such that
$\hat{\pi}_{1} (\phi \overline{\langle G_{u}, H \rangle})=1$.
We have $\phi_{1} \in  {\langle G_u, H \rangle}_{1}$.
Since ${\langle G_{u}, H \rangle}_{1}= \langle G_{1,u}, H_{1} \rangle$ and
$H_{1}$ is normal in $G_1$, $\phi_1$ is of the form $\alpha \beta$ where
$\alpha \in G_{1,u}$ and $\beta \in H_1$. Let $\psi$ be an element of
$\overline{H}$ such that $\psi_{1} = \beta$.
The formal diffeomorphism $\psi^{-1} \circ \phi$ is unipotent and hence
contained in $\overline{G}_{u}$. Since
$\overline{G}_{u} \subset \overline{\langle G^{u}, H \rangle} =
\overline{\langle G_{u}, H \rangle}$ by Proposition \ref{pro:zuuz} and Equation
(\ref{equ:forudua}), we obtain $\psi^{-1} \circ \phi \in \overline{\langle G_{u}, H \rangle}$.
It is clear that $\phi$ belongs to $\overline{\langle G_{u}, H \rangle}$
since $\psi \in \overline{H}$. Hence $\hat{\pi}_{1}$ is injective.
In particular $G/\langle G_{u}, H \rangle$ is finite dimensional by Proposition
\ref{pro:elem}.
More precisely we get $\dim G/\langle G_{u}, H \rangle =
\dim G_1 - \dim {\langle G_{u}, H \rangle}_1$.
\end{proof}
%\begin{cor}
%Let $G$ be a splittable virtually solvable subgroup of $\diffh{}{n}$.
%Then $G/G_{u}$ is finite dimensional. In particular $G$ is finite
%dimensional if and only if $G_{u}$ is finite dimensional.
%\end{cor}
%\begin{proof}
%The set $G_{u}$ is a subgroup of $G$ by Proposition \ref{pro:zuuz}.
%Since $G$ is splittable, we obtain $G_u = G^u = \langle G^u \rangle$.
%Thus $\overline{G}_{u} = \overline{(G_u)}$ by
%Proposition \ref{teo:gfdgufd}.
%We denote $H=G_u$. Fix $k \in {\mathbb N}$. Consider the map
%$\hat{\pi}_{1}: \overline{G}/\overline{H} \to G_{1}/H_{1}$ induced by
%$\pi_{1}: \overline{G} \to G_1$.
%Let $\phi \overline{H}$ be an element of $\ker (\hat{\pi}_{1})$.
%It satisfies $\pi_{k}(\phi) \in H_{1}$.
%Since $H \subset  \overline{G}_{u}$ and $\overline{G}_{u}$ is a pro-algebraic
%group consisting of unipotent elements, all elements of $H_{1}$ are unipotent.
%Therefore $\phi$ is unipotent
%and $\phi \in \overline{G}_{u} =\overline{H}$.
%The kernel of $\hat{\pi}_{1}$ is the trivial group and as a consequence
%$\hat{\pi}_{1}$ is an isomorphism of groups.
%We obtain that $G/G_u$ is finite dimensional by Proposition \ref{pro:elem}.
%More precisely we get $\dim G/G_u = \dim G_1 - \dim G_{u,1}$.
%\end{proof}
\begin{teo}
\label{teo:rgwgu}
Let $H$ be a normal subgroup of a subgroup $G$ of $\diffh{}{n}$.
Suppose that
\begin{itemize}
\item $G/H$ is virtually solvable,
\item $G/ \langle H, G_u \rangle$ is finitely generated in the extended sense and
\item either $G_{1}^{*}$ or $H_{1}^{*}$ is algebraic.
\end{itemize}
Then $G/\langle G_{u}, H \rangle$ is finite dimensional.
\end{teo}
\begin{proof}
Let us show $\dim G/\langle G_{u}, H \rangle < \infty$ under the following hypotheses:
\begin{itemize}
\item There exists a normal subgroup $J$ of $G$ such that $J \subset H$ and $G/J$ is virtually
solvable,
\item $G/ \langle H, G_u \rangle$ is finitely generated in the extended sense and
\item $J_{1}^{*}$ is algebraic.
\end{itemize}
There exists a subgroup $G_{+}$ of $G$ such that $ \langle H, G_u \rangle \subset G_{+}$,
$\overline{G}_{+}=\overline{G}$ and $G_{+}/ \langle H, G_u \rangle$ is finitely generated.
Up to replace $G$ with $G_{+}$ we can suppose that
$G / \langle H, G_u \rangle$ is finitely generated.

%The group $\overline{G}/\overline{J}$ is virtually solvable by Lemma \ref{lem:claia}.
Since $G_{1}^{*}/J_{1}^{*}$ and $\langle G_{1}^{*}, J_{1} \rangle /J_{1}$ are virtually solvable,
$G_{1}/J_{1}$ is also virtually solvable.
Consider the natural maps $\hat{\tau}_{1}: \overline{G} \to G_{1}/J_{1}$ and
$\hat{\pi}_{1}: \overline{G}/\overline{J} \to G_{1}/J_{1}$.
%Since $\hat{\pi}_{1}$ is surjective, $G_{1}/J_{1}$ is a virtually solvable algebraic matrix group.
There exists a connected finite index normal subgroup $S/J_{1}$ of $G_{1}/J_{1}$.
We define $T= (\hat{\tau}_{1})^{-1}(S/J_{1}) \cap G$. It is a finite index normal subgroup of $G$
containing $J$.
Since $S/J_{1}$ is connected and virtually solvable, it is solvable.
In particular  $S/J_{1}$ is triangularizable by Lie-Kolchin's theorem.
We deduce that the elements of
the derived group $(S/J_{1})'$ are unipotent.
Therefore $\hat{\tau}_{1}(T')$ is a subgroup of $G_{1}/J_{1}$ of unipotent elements.
Given any $\phi \in T'$ we have that $\phi_{1}$ is of the form $\alpha \beta$ with
$\alpha \in G_{1,u}$ and $\beta \in J_{1}$.
Since $J_{1} = J_{1}^{*}$, there exists $\psi \in J$ such that $\psi_{1} = \beta$.
The formal diffeomorphism $\psi^{-1} \circ \phi$ belongs to $T_u$.
We deduce
\begin{equation}
\label{equ:dtgu}
T' \subset \langle T_{u}, J \rangle \subset \langle G_u, H \rangle. 
\end{equation}
The extension $G/\langle T, G_{u}, H \rangle$ is finite.
It is finite dimensional by Lemma \ref{lem:fis}.
Since $\langle T, G_{u}, H \rangle/ \langle G_u, H \rangle$ is a finite index subgroup
of the finitely generated group $G/ \langle G_u, H \rangle$, it is finitely generated
(cf.  \cite[Theorem 1.6.11]{Robinson}).
Moreover $\langle T, G_{u}, H \rangle/ \langle G_u, H \rangle$ is abelian by 
Property (\ref{equ:dtgu}). Proposition \ref{pro:fgaifd} implies that
$\langle T, G_{u}, H \rangle/ \langle G_u, H \rangle$
is finite dimensional. Thus $G/ \langle G_u, H \rangle$
 is finite dimensional by Proposition \ref{pro:elems}.

By defining $J=H$ the case where $H_{1}^{*}$ is algebraic is proved. Let us
suppose that $G_{1}^{*}$ is algebraic.
Let $M$ be a finite index normal subgroup of $G$ containing $H$ and such
that $M/H$ is solvable. There exists a derived group $M^{(\ell)}$ of $M$ contained
in $H$. We define $J= M^{(\ell)}$. Since derived groups are characteristic, $J$
is a normal subgroup of $G$. It is clear
that $J \subset H$ and $G/J$ is virtually solvable.
The group $M_{1}^{*}$ is a finite index normal subgroup of $G_{1}^{*}$
and since the latter group is algebraic, $M_{1}^{*}$ is algebraic.
Moreover $J_{1}^{*}$ is the $\ell$-th derived group of $M_{1}^{*}$.
Derived groups of algebraic groups are algebraic (cf.  \cite{Borel}[2.3, p. 58]),
thus $J_{1}^{*}$ is algebraic.
\end{proof}

\bibliography{rendu}
\end{document}